\documentclass[svgnames]{amsart}
\usepackage[english]{babel}
\usepackage[utf8]{inputenc}

\usepackage[margin=3cm]{geometry}

\usepackage[colorlinks=true, pdfstartview=FitV, linkcolor=black, citecolor=black, urlcolor=blue]{hyperref}
\usepackage[nameinlink]{cleveref}

\usepackage{amssymb, amsfonts, amsmath, amsthm, amscd, amsxtra, latexsym, mathabx, mathtools, enumitem}
\usepackage{tikz}\usetikzlibrary{cd}
\usepackage{epigraph}

\theoremstyle{plain}
\newtheorem{thm}{Theorem}[section] 
\newtheorem{lem}[thm]{Lemma} 
\AddToHook{env/lem/begin}{\crefalias{thm}{lem}}
\newtheorem{claim}[thm]{Claim}
\AddToHook{env/claim/begin}{\crefalias{thm}{claim}}
\newenvironment{claimproof}{\begin{proof}}{\end{proof}}
\newtheorem{prop}[thm]{Proposition}
\AddToHook{env/prop/begin}{\crefalias{thm}{prop}}
\newtheorem{cor}[thm]{Corollary}
\AddToHook{env/cor/begin}{\crefalias{thm}{cor}}

\newtheorem{thmintro}{Theorem}

\newtheorem{corintro}[thmintro]{Corollary}
\AddToHook{env/corintro/begin}{\crefalias{thmintro}{corintro}}
\newtheorem{questintro}[thmintro]{Question}
\AddToHook{env/questintro/begin}{\crefalias{thmintro}{questintro}}

\theoremstyle{definition}
\newtheorem{defn}[thm]{Definition}
\AddToHook{env/defn/begin}{\crefalias{thm}{defn}}

\AddToHook{env/ass/begin}{\crefalias{thm}{ass}}
\newtheorem{notation}[thm]{Notation}
\AddToHook{env/notation/begin}{\crefalias{thm}{notation}}
\newtheorem{rem}[thm]{Remark}
\AddToHook{env/rem/begin}{\crefalias{thm}{rem}}
\newtheorem{ex}[thm]{Example}
\AddToHook{env/ex/begin}{\crefalias{thm}{ex}}

\AddToHook{env/construction/begin}{\crefalias{thm}{construction}}

\AddToHook{env/remintro/begin}{\crefalias{thmintro}{remintro}}

\newcounter{gcomments}

\newcounter{ecomments}

\newcommand{\Aut}[1]{\operatorname{Aut}\left(#1\right)}
\newcommand{\AutH}[2]{\operatorname{Aut}_{#2}\left(#1\right)}
\newcommand{\AutbH}[2]{\operatorname{Aut}_{#2}^\partial\left(#1\right)}
\newcommand{\Out}[1]{\operatorname{Out}\left(#1\right)}
\newcommand{\Outstar}[1]{\operatorname{Out}^\star\left(#1\right)}
\newcommand{\TOut}[1]{\widetilde{\operatorname{Out}}\left(#1\right)}
\newcommand{\Outtwo}[1]{\operatorname{Out}_2\left(#1\right)}
\newcommand{\Inn}[1]{\operatorname{Inn}\left(#1\right)}

\newcommand{\piorb}[1]{\pi_1^{\text{orb}}\left(#1\right)}
\newcommand{\MCG}[1]{\operatorname{MCG}\left(#1\right)}
\newcommand{\EMCG}[1]{\operatorname{MCG}^\pm\left(#1\right)}

\newcommand{\EPMCG}[1]{\operatorname{PMCG}^\pm\left(#1\right)}
\newcommand{\PMCG}[1]{\operatorname{PMCG}\left(#1\right)}
\newcommand{\MCGb}[1]{\operatorname{MCG}^\partial\left(#1\right)}

\newcommand{\frakS}{\mathfrak S}
\newcommand{\C}[1]{\mathcal{C}(#1)}
\newcommand{\propnest}{\sqsubsetneq}

\newcommand{\trans}{\pitchfork}
\newcommand{\nest}{\sqsubseteq}

\newcommand{\N}{\mathbb{N}}
\newcommand{\Z}{\mathbb{Z}}
\newcommand{\R}{\mathbb{R}}
\newcommand{\Hyp}{\mathbb{H}^2}

\newcommand{\dist}{\mathrm{d}}
\newcommand{\diam}{\mathrm{diam}}
\newcommand{\Stab}[2]{\operatorname{Stab}_{#1}\left(#2\right)}

\title[Outer automorphism groups, bounded extensions, and HHG]{Outer automorphism groups of hyperbolic groups, bounded extensions, and hierarchical hyperbolicity}

\author[E. Had\v{z}iosmanovi\'{c}]{Ervin Had\v{z}iosmanovi\'{c}}
    \address{(Ervin Had\v{z}iosmanovi\'{c}) Scuola Normale Superiore, Pisa, Italy}
    \email{ervin.hadziosmanovic@sns.it}

\author[G. Mangioni]{Giorgio Mangioni}
    \address{(Giorgio Mangioni) Maxwell Institute and Department of Mathematics, Heriot-Watt University, Edinburgh, UK}
    \email{gm2070@hw.ac.uk}

\begin{document}
\begin{abstract}
    We prove that the outer automorphism group of a one-ended hyperbolic group is virtually a hierarchically hyperbolic group (HHG), under mild orientability conditions on the associated JSJ decomposition. This is done by proving that a finite-index subgroup is a central extension of a product of orbifold mapping class groups, and the extension has bounded Euler class. Our theorem is sharp: we exhibit a surface amalgam whose fundamental group has full outer automorphism group which is not a HHG. 
\end{abstract}
\maketitle

\epigraph{This terminology should not be blamed on me. It was obtained by a democratic process in my course of 1976–77. An orbifold is something with many folds; unfortunately, the word ``manifold" already has a different definition. I tried ``foldamani", which was quickly displaced by the suggestion of ``manifolded". After two months of patiently saying ``no, not a manifold, a manifold\emph{ed}," we held a vote, and ``orbifold" won.}{William Thurston, \emph{The geometry and topology of three-manifolds}}

\section*{Introduction}
Outer automorphism groups of hyperbolic groups are of utmost importance in combinatorial and geometric group theory. One of the many reasons is that this class includes both $\Out{\mathbb{F}_n}$ and mapping class groups of closed surfaces, which are ubiquitous in the study of low-dimensional topology. Furthermore, a key step in the solution of the isomorphism problem for hyperbolic groups (see, among others, \cite{Rips_Sela,Sela,DahmaniGuirardel_Isoprob}) was a description of the outer automorphism group of a one-ended hyperbolic group $G$ in terms of certain virtually Abelian extensions of mapping class groups. We improve this description by showing that $\Out{G}$ is virtually a \emph{bounded} central extension:

\begin{thmintro}[See \Cref{thm:M}]\label{thmintro:virtual_bounded}
    Let $G$ be a one-ended hyperbolic group. Then $\Out{G}$ is virtually a direct product $\Z^q\times M$, where $q\ge 0$ and $M$ fits inside a bounded central extension 
    \[1\to Z_s \to M\to \prod \PMCG{G_v}\to 1.\]
    Here $Z_s$ is finitely generated Abelian, the product has finitely many factors, and each $\PMCG{G_v}$ is a finite extension of a finite-index subgroup of the mapping class group of a hyperbolic orbifold without mirrors.
\end{thmintro}

\noindent Let us briefly explain the above terminology. We say that a group $G$ is a \emph{finite extension} of a group $H$ if $G$ surjects onto $H$ with finite kernel. Furthermore, a central extension $1\to K\to G\to H\to 1$ is \emph{bounded} if the associated Euler class $[G]\in H^2(H,K)$ is represented by a cocycle $\omega \colon H\times H\to K$ taking bounded values; see \Cref{bounded via qm:Abelian_ker} for an equivalent characterisation in terms of \emph{quasimorphisms}. Finally, we refer to \Cref{subsec:Orbifolds} below for more details on orbifolds. For experts, the product is over all \emph{quadratically hanging subgroups} $G_v$ appearing in the \emph{JSJ tree of cylinders} over the class $\mathcal Z$ of virtually $\Z$ subgroups with infinite center: see \Cref{subsec:jsj} and the outline below.

If we further assume that $G$ is torsion-free, then all orbifolds are genuine surfaces, and all finite extensions are trivial. Hence in this case \Cref{thmintro:virtual_bounded} has a simpler form: 

\begin{corintro}
    Let $G$ be a one-ended, torsion-free hyperbolic group. Then $\Out{G}$ is virtually a bounded central extension of a direct product of mapping class groups of (possibly non-orientable) hyperbolic surfaces.
\end{corintro}

\noindent While the existence of the central extension from \Cref{thmintro:virtual_bounded} follows from a careful inspection of several, often implicit, results in the literature (see, among others, \cite{Sela,Bowditch_JSJ,Levitt_authyp,DahmaniGuirardel_Isoprob}), the fact that the extension is bounded is new. To the eye of a geometric group theorist, one of the main consequences of boundedness is that the extension is \emph{quasi-isometrically trivial}, meaning that it is quasi-isometric to the direct product of the kernel and the quotient \cite{Gersten_B_is_QIT}. In the case where all orbifolds appearing in \Cref{thmintro:virtual_bounded} are \emph{orientable}, we can dramatically improve our understanding of the coarse geometry of $\Out{G}$:

\begin{thmintro}[See \Cref{thm:Out(Gone_end_hyp)_is_HHG}]\label{thmintro:virtual_HHG}
    Let $G$ be a one-ended hyperbolic group, and let $\mathcal Z$ be the family of its virtually $\Z$ subgroups with infinite centre. Suppose that all quadratically hanging subgroups appearing in the $\mathcal Z$-JSJ tree of cylinders of $G$ are orientation-preserving. Then $\Out{G}$ is virtually a hierarchically hyperbolic group.
\end{thmintro}

\noindent \emph{Hierarchically hyperbolic groups} (HHGs), as first introduced in \cite{HHS_I}, provide a common framework for mapping class group of surfaces, most cubulated groups, most 3-manifold groups, and more. As such, techniques from low-dimensional topology and the world of CAT(0) cube complexes can be exploited to prove a plethora of properties of (virtual) HHGs; a very non-exhaustive list includes results on their asymptotic dimension \cite{asdim}, quasiflats \cite{quasiflat}, and on being quasi-isometric to CAT(0) cube complexes \cite{petyt}.

We emphasise that the splitting from \Cref{thmintro:virtual_bounded} can be trivial: this happens precisely when $G$ is itself a finite extension of an \emph{orientation-preserving, cocompact Fuchsian group of conical type} (that is, the fundamental group of a closed, orientable, hyperbolic $2$-orbifold). In this case $\Out{G}$ is genuinely hierarchically hyperbolic, without need to pass to a finite-index subgroup:

\begin{thmintro}[See \Cref{thm:Out(Vsurface)_is_HHG}]\label{thmintro:out_cocompact_is_HHG}
     Let $G$ be a finite extension of an orientation-preserving, cocompact Fuchsian group of conical type. Then $\Out{G}$ is a hierarchically hyperbolic group.
\end{thmintro}

\noindent As the full $\Out{G}$ is quasi-isometric to its finite-index HHG subgroup, it has the structure of a hierarchically hyperbolic \emph{space}, the non-equivariant analogue of a HHG. Unfortunately, being a hierarchically hyperbolic \emph{group} might not pass to finite-index overgroups: in other words, there might not be any equivariant structure (see \cite[Corollary 4.5]{PetytSpriano_Eyries} for an explicit counterexample). This is indeed the case in our setting, even under the assumption that $G$ has no torsion. For a concrete counterexample, let $\Sigma=S_1^1$ be a torus with one open disk removed, and let $K$ be the $2$-complex obtained by gluing three copies $\Sigma_1$, $\Sigma_2$, $\Sigma_3$ along the boundary circle $C$, as in \Cref{fig:three_tori}.

\begin{thmintro}[See \Cref{thm:threesurfaces}]\label{thmintro:not_HHG}
    $\pi_1(K)$ is torsion-free and satisfies the assumptions of \Cref{thmintro:virtual_HHG}, but $\Out{\pi_1(K)}$ is not a HHG.
\end{thmintro}

\begin{figure}[htp]
    \centering
    \includegraphics[width=0.75\linewidth]{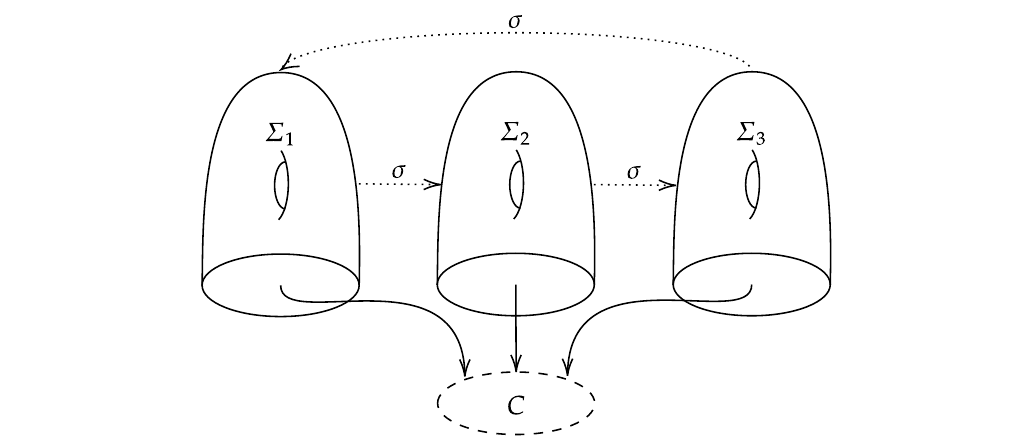}
    \caption{Depiction of the complex $K$, obtained by gluing three tori with one boundary to a common circle (here, the dashed line) along the boundaries. The cyclic permutation of the three surfaces descends to an order-three element $\sigma$ of $\Out{\pi_1(K)}$.}
    \label{fig:three_tori}
\end{figure}

\subsection*{Outline of sections and arguments}
\subsubsection*{JSJ decompositions}
The main tool for the study of outer automorphism groups of hyperbolic groups, which also played a central role in the solution of the isomorphism problem for hyperbolic groups, is Bowditch's \emph{JSJ decomposition} over virtually cyclic subgroups \cite{Bowditch_JSJ}. This is a graph-of-groups decomposition of a one-ended hyperbolic group with virtually cyclic edge groups, and whose vertex groups are either:
\begin{itemize}
    \item maximal virtually cyclic subgroups;
    \item ``rigid" quasiconvex subgroups, which cannot be split further relative to the edge groups;
    \item or \emph{quadratically hanging subgroups}, which are (finite extensions of) fundamental groups of hyperbolic orbifolds.
\end{itemize}
The decomposition is invariant under automorphisms of $G$, in a suitable sense, and this makes it possible to study automorphisms of $G$ by their action on the vertex groups. 

We do not directly use Bowditch's decomposition; instead, following \cite{DahmaniGuirardel_Isoprob}, we manipulate it to get another canonical decomposition, the \emph{$\mathcal Z$-JSJ tree of cylinders}. The new splitting has all properties of Bowditch's decomposition, and the advantage that quadratically hanging subgroups are now extensions of fundamental groups of hyperbolic orbifolds \emph{without mirrors}, i.e. which don't contain reflections along lines of $\Hyp$. This is relevant as mapping class groups of such orbifolds are finite-index subgroups of mapping class groups of surfaces, and this is the key fact underlying the proof of \Cref{thmintro:virtual_HHG}. 

In order to uniformise the notation among the numerous references on the subject, we did an extensive literature review which we then condensed in \Cref{sec:Background}. There the reader will find all the relevant background material on (bounded) central extensions, orbifolds, their mapping class groups, and JSJ decompositions of hyperbolic groups. While we do not claim any completeness of our exposition, as we omitted several facts that we do not need in this paper, we still hope it might serve as a gentle, self-contained entry point for the reader interested in any of the above topics.

\subsubsection*{(Virtual) bounded central extensions}
\Cref{sec:virtual_HHG} contains the virtual description of the outer automorphism group as a bounded central extension of a product of mapping class groups, which proves \Cref{thmintro:virtual_bounded}. Though a sizeable part of what we show is implicit in work of Levitt \cite{Levitt_authyp}, the proof of boundedness is new, and relies on a careful analysis of the analogue of the capping extension for orbifold mapping class groups. 

We now illustrate the general idea with a simple example. Take the complex $K$ depicted in \Cref{fig:three_tori} and let $G$ be its fundamental group. By an analogue of the Dehn-Nielsen-Baer Theorem \cite[Theorem 1.2]{Lafont}, automorphisms of $G$ are realised by self-homeomorphisms of $K$; hence restricting a homeomorphism to (the interior of) each $\Sigma_i$ gives an epimorphism $\rho\colon \TOut{G}\to \prod_{i=1}^3\MCG{S_{1,1}}$, where $\TOut{G}$ is the finite-index subgroup of $\Out{G}$ mapping each torus to itself in an orientation-preserving way, and $S_{1,1}$ is a torus with one puncture. The kernel of $\rho$ is generated by the \emph{Dehn twists} $T_1,T_2,T_3$ around the boundaries of the $\Sigma_i$'s; these elements span a subgroup isomorphic to $\Z^2$ as the product $T_1T_2T_3$ induces the conjugation by the core curve of $C$ on $G$. Since homeomorphisms of $K$ fix $C$ up to homotopy, Dehn twists are central in $\TOut{G}$, so we get a central extension
\[1\to \Z^2\to \TOut{G}\to \prod_{i=1}^3\MCG{S_{1,1}}\to 1.\]
The boundedness of the above extension reduces to the boundedness of the classical \emph{capping extension}
\[1\to \langle T\rangle \to \MCG{S_1^1}\to \MCG{S_{1,1}}\to 1,\]
where again the map on the right is the restriction to the interior; in turn, this follows from the existence of a certain quasimorphism $\MCG{S_1^1}\to \langle T\rangle\cong \Z$, obtained from the action on the \emph{annular curve graph} of the boundary \cite[Section 2.4]{masurminsky2}.

We stress that the capping extension can be shown to be bounded by other means: for instance, one could combine the boundedness of the classical Euler class \cite[Sections 10.2, 10.3]{frigerio} with the results in \cite[Section 5.5]{FM}. Such an argument could likely be extended to general orbifolds, but we preferred to use quasimorphism, as our techniques are more elementary.

\subsubsection*{Hierarchical hyperbolicity} Very roughly, a \emph{hierarchically hyperbolic group} is a finitely generated group admitting suitable ``coordinate projections'' to a collection of Gromov-hyperbolic spaces: see \Cref{defn:HHS} and \Cref{defn:HHG} below. The motivating example is the mapping class group of a finite-type, orientable surface, which projects to the curve graphs of its subsurfaces via subsurface projection. Combining this with the fact that bounded central extensions of HHGs are themselves HHGs \cite[Theorem 1.3]{FFMS}, in \Cref{subsec:virtual_HHG_of_Out} we deduce \Cref{thmintro:virtual_HHG} from \Cref{thmintro:virtual_bounded}. In the case where $G$ is itself a finite extension of a cocompact Fuchsian group of conical type, the same techniques describe $\Out{G}$ as a finite extension of a finite-index subgroup of a mapping class group, thus proving \Cref{thmintro:out_cocompact_is_HHG} without the need to pass to a finite-index subgroup. We strived to make \Cref{subsec:virtual_HHG_of_Out} as clear as possible for the reader who is not an expert on HHG: the only prerequisites we need are given as self-contained blackboxes, while a comprehensive background on HHGs is postponed to \Cref{subsec:HHS}.

\subsubsection*{The counterexample}
Finally, in \Cref{sec:counterex} we construct the counterexample, thus proving \Cref{thmintro:not_HHG}. Very roughly, assuming that $\Out{\pi_1(K)}$ is a HHG, we show that the $(3,3,3)$ triangle group $\Z^2\rtimes S_3\le \Out{\pi_1(K)}$, generated by Dehn twists around the boundaries of the $\Sigma_i$'s and the homeomorphisms permuting the three tori, is \emph{hierarchically quasiconvex}, the suitable notion of quasiconvexity in the hierarchical world. In analogy with quasiconvex subgroups in other non-positively curved groups, hierarchically quasiconvex subgroups inherit some of the geometric structure of the ambient group; we use this to deduce a contradiction, mimicking an argument of Petyt-Spriano who proved that the $(3,3,3)$ triangle group carries no HHG structure (see \cite[Corollary 4.5]{PetytSpriano_Eyries}).  

The proof of \Cref{thmintro:not_HHG} relies on the technical \Cref{prop:quasiflat_sgr_are_std}. For HHG experts, the latter states that an Abelian subgroup $A$ of maximal rank in a HHG coincides with a standard flat. The first version of this paper deduced this fact from the Quasiflat theorem \cite[Theorem A]{quasiflat}, using that the cocompact action of $A$ onto itself forbids the presence of ``singular points" in the quasiflat. Later, Pénélope Azuelos pointed out to us that the result is actually a straightforward consequence of \cite[Proposition 2.17]{HRSS}. We therefore decided to shorten our argument, and we refer anyone who might be interested in the original tools to the first ArXiv version of this article.

\subsection*{Questions}
As explained in \Cref{rem:where_orientable}, we could remove the orientability assumption in \Cref{thmintro:virtual_HHG}, and therefore prove virtual hierarchical hyperbolicity for outer automorphism of all one-ended hyperbolic groups, if the following question had a positive answer:
\begin{questintro}\label{questintro:MCG_nonor_orb}
    Let $O$ be a non-orientable, hyperbolic orbifold without mirrors. Is the mapping class group of $O$ a hierarchically hyperbolic group?
\end{questintro}
\noindent As pointed out in \Cref{lem:just_surface_mcg}, $\MCG{O}$ has finite index in the mapping class group of the surface obtained by removing the cone points. Since being a hierarchically hyperbolic group passes to finite-index subgroups, it would be enough to answer the following:
\begin{questintro}\label{questintro:MCG_nonor_surf}
    Is the mapping class group of a non-orientable surface, possibly with punctures and boundaries, a hierarchically hyperbolic group?
\end{questintro}
\noindent To the best of our knowledge, this is unknown even for closed non-orientable surfaces.

A positive answer to \Cref{questintro:MCG_nonor_surf} would follow from a more general result about centralisers in HHGs. Indeed, let $N$ be a non-orientable surface, let $S$ be its orientable double cover, and let $J\colon  S\to S$ be the order-two deck transformation. By combining results of \cite{Birman_chillingworth,HopeTillmann,Lima_Guaschi_Maldonado,Katayama_Kuno} one gets that, except in a few sporadic cases, $\MCG{N}$ embeds in $\EMCG{S}$ as the subgroup of orientation-preserving mapping classes that commute with $J$ (see \Cref{thm:orientable_double_cover_centraliser} below). So we ask:
\begin{questintro}\label{quest:finite_centraliser}
    Let $G$ be a HHG, and let $F\le G$ be a finite subgroup. Is the centraliser of $F$ a HHG?
\end{questintro}
\noindent The above is motivated by the principle that centralisers in non-positively curved groups should inherit the geometric properties of the ambient group: for example, centralisers in hyperbolic groups are quasiconvex, hence hyperbolic (see e.g. \cite[Proposition III.$\Gamma$.3.9]{BridsonHaefliger}), and centralisers in CAT(0) groups are themselves CAT(0) \cite{Ruane}. The analogue of \Cref{quest:finite_centraliser} for centralisers of virtually Abelian subgroup will be considered in forthcoming work of Azuelos and Hagen \cite{HHS_flat_torus}.

\subsection*{Acknowledgements}
We thank Alessandro Sisto for hosting EH during his research stay at the Maxwell Institute in Edinburgh, where this work was written, for asking the questions we answer, and for several insightful comments and suggestions. We are grateful to Pénélope Azuelos and Mark Hagen for numerous exchanges, and for suggesting how to shorten the proof of the counterexample. Many thanks to Alex Wright for sharing his ideas on how to tackle \Cref{quest:finite_centraliser}. We also acknowledge Gilbert Levitt for enlightening discussion on the technicalities hidden in defining mapping class groups for orbifolds. Finally, the first author was partially supported by the GNSAGA group of INdAM.

\section{Background}\label{sec:Background}

\subsection{Bounded central extensions}\label{sec:bdd_ext}
Here we gather generalities on (bounded) central extensions, and recall a characterisation of boundedness in terms of quasihomomorphisms: see \Cref{bounded via qm:Abelian_ker}.

\begin{defn} Given a central extension $1\to K\to E\to G\to 1$ with finitely generated kernel, let $\sigma \colon G \to E$ be a set-theoretic section. The map $\omega(g, h) = \sigma(g)\sigma(h)\sigma(gh)^{-1}\in K$ is a \emph{$2$-cocycle}, i.e. for all $g_1,g_2,g_3\in G$
\[\omega(g_2, g_3) - \omega(g_1g_2, g_3) + \omega(g_1, g_2g_3) - \omega(g_1, g_2) = 0.\]
Hence $\omega$ defines a class in $H^2(G; K)$, called the \emph{Euler class} of the extension, which is independent of the choice of $\sigma$ and we denote by $[E]$. We say the extension is \emph{bounded} if $[E]$ is bounded, meaning that there exists a $2$-cocycle $\alpha\colon G^2\to K$ such that $[\alpha]=[E]$ and $\alpha$ takes finitely many values.
\end{defn}

\noindent Bounded extensions are quasi-isometrically trivial, in the following sense:
\begin{thm}[{\cite{Gersten_B_is_QIT}}]\label{thm:gersten}
    Let $1\to K\to E\to G\to 1$ be a bounded central extension. Then $E$ is quasi-isometric to the direct product $K\times G$.
\end{thm}

\begin{defn}
\label{def:qm}
    Let $(K, \| \cdot \|)$ be a normed Abelian group, and let $E$ be another group. A map $\chi\colon E\to K$ is a \emph{quasihomomorphism} if there exists $D(\chi)\ge 0$, called the \emph{defect} of $\chi$, such that, for every $e_1,e_2\in E$, 
    \[\|\chi(e_1)+\chi(e_2)-\chi(e_1e_2)\|\le D(\chi).\]
    When $K=\Z$ or $\R$ with the Euclidean norm, we say \emph{quasimorphism} instead of quasihomomorphism.  A quasihomomorphism is \emph{homogeneous} if it restricts to a homomorphism on every cyclic subgroup.
\end{defn}

\noindent It is clear that a bounded modification of a quasihomomorphism is itself a quasihomomorphism, possibly with a larger defect. We shall often use this fact without mentioning it.

\begin{rem}\label{rem:homogeneisation}
    Given any quasihomomorphism $\phi\colon E\to \R^n$, the map $\psi(e)=\lim_{k\to \infty}\frac{\phi(e^k)}{k}$, called the \emph{homogeneisation} of $\phi$, is a homogeneous quasihomomorphism within finite distance from $\phi$ (see e.g. \cite[Lemma 2.21]{scl}). Notice that a homogeneous quasimorphism $\psi$ coincides with its homogeneisation, and is therefore invariant under conjugation: given $g,h\in E$, 
    \[\psi(ghg^{-1})=\lim_{k\to \infty}\frac{\psi(gh^kg^{-1})}{k}=\lim_{k\to \infty}\frac{\psi(g)+\psi(h^k)+\psi(g^{-1})}{k}=\psi(h),\]
    where we used that $\psi(gh^kg^{-1})$ differs from the sum of the $\psi$-images by at most $2D(\psi)$.
\end{rem}

\noindent Natural examples of quasimorphisms arise from actions on quasilines:
\begin{ex}[{Busemann quasimorphism, see e.g. \cite[Section 4.1]{Manning_actions_on_hyp}}]\label{ex:busemann}
    Let $E$ be a group acting on a \emph{quasiline} $X$, i.e. a space quasi-isometric to $\R$. Suppose the action does not swap the ideal endpoints. Given a sequence $\{x_n\}\in X$ converging to an ideal endpoint $p$, let $\phi_{\{x\}}\colon E\to \R$ be defined as
    \[\phi_{\{x\}}(g)=\limsup_{n\to \infty} \dist(gx_0,x_n)-\dist(x_0,x_n),\]
    and let $\phi(g)=\lim_{k\to \infty}\frac{\phi_{\{x\}}(g^k)}{k}$. One can show that the limit exists, does not depend on the choice of the sequence, and gives a homogeneous quasimorphism $\phi\colon E\to \R$ which is non-trivial on $g$ if and only if $g$ acts loxodromically on $X$.
\end{ex}

\noindent We recall the following equivalent characterisation of boundedness for central extensions, which is very useful for applications: see \cite[Proposition 2.9]{FFMS} and, independently, \cite[Lemma 5.6]{tao2025properactionsfiniteproducts}.
\begin{prop}\label{bounded via qm:Abelian_ker}
    Let $1 \to K \to E \xrightarrow{\pi} G \to 1$ be a central extension with finitely generated kernel. The following are equivalent:
    \begin{enumerate}
        \item $[E]$ is bounded.
        \item There exists a quasihomomorphism $\chi \colon E \to K$ such that $\chi|_K$ is the identity.
    \end{enumerate}
\end{prop}
\noindent We give some easy applications of the above criterion, which are surely known to experts via other means. The first regards products of central extensions:
\begin{cor}\label{cor:prod_of_bounded_is_bounded}
    For $i=1,2$ let $1 \to K_i \to E_i \to G_i \to 1$ be bounded central extensions with finitely generated kernels. Then the direct product $1 \to K_1\times K_2 \to E_1\times E_2 \to G_1\times G_2 \to 1$ is a bounded central extension.
\end{cor}
\begin{proof}
    If for $i=1,2$ we choose a quasimorphism $\psi_i\colon E_i\to K_i$ which is the identity on $K_i$, then $\psi_1\times\psi_2\colon E_1\times E_2\to K_1\times K_2$ is a quasihomomorphism which is the identity on $K_1\times K_2$.
\end{proof}
\noindent The second consequence regards sub-extensions:
\begin{cor}\label{cor:subext_of_bounded_is_bounded}
    Let $1 \to K \to E \xrightarrow{\pi} G \to 1$ be bounded central extensions with finitely generated kernel, and let $E'\le E$. Then the central extension $1 \to K\cap E' \to E' \to \pi(E')\to 1$ is bounded.
\end{cor}
\begin{proof}
    Let $\psi\colon E\to K$ be a quasihomomorphism which is the identity on $K$. The restriction $\psi|_{E'}\colon E'\to K$ is a quasihomomorphism which is the identity on $K\cap E'$. Since $K$ is an abelian group, we can then find a quasihomomorphism $\xi\colon K\to K\cap E'$ which is the identity on $K\cap E'$. Hence the composition $\xi\circ \psi|_{E'}\colon E'\to E'\cap K$ is a quasihomomorphism which is the identity on $K\cap E'$, proving that the extension is bounded. 
\end{proof}

\noindent The last consequence was already pointed out in \cite[Corollary 2.10]{FFMS}:
\begin{cor}\label{cor:quot_of_bounded_by_kernel}
Let $1 \to K \to E\to G \to 1$ be a bounded central extension with finitely generated kernel. If $L\le K$, then the extension $1 \to K/L \to E/L\to G \to 1$ is also bounded.
\end{cor}

\subsection{Hyperbolic orbifolds and their mapping class groups}\label{subsec:Orbifolds}
We recall here definition and properties of \emph{hyperbolic 2-orbifolds} and their \emph{mapping class groups}, gathering notation and results from across the literature for improved clarity. We refer to \cite{Maclachlan_OrbifoldMCG, Bowditch_JSJ,Fujiwara, DahmaniGuirardel_Isoprob} (among several others) for the original treatments, and to \cite{Martelli} for a more recent, comprehensive introduction.

\subsubsection{Bounded Fuchsian groups}
In what follows, let $\Hyp$ be the hyperbolic plane, and let $\partial\Hyp$ its boundary at infinity, which is homeomorphic to a circle $\mathbb{S}^1$.
\begin{rem}[Isometries of $\Hyp$]\label{rem:types}
An isometry $g\in \operatorname{Isom}(\Hyp)$ is of one of $5$ \emph{types}. If $g$ is orientation-preserving, then it is either:
\begin{itemize}
    \item\emph{elliptic} if it fixes a point in $\Hyp$;
    \item\emph{parabolic} if it fixes exactly one point of the boundary $\partial\Hyp$;
    \item\emph{loxodromic} if it fixes exactly two ideal points of $\partial \Hyp$ and acts as a translation along the unique geodesic that connects them, called the axis of $g$.
\end{itemize}
If instead $g$ reverses orientation, then it is either:
\begin{itemize}
    \item a \emph{reflection} along a geodesic;
    \item a \emph{glide-reflection} along a geodesic, that is, the composition of a reflection along a geodesic and a loxodromic along the same geodesic.
\end{itemize}
The subgroup of $\operatorname{Isom}(\Hyp)$ fixing a point is isomorphic to the orthogonal group $O(2)$ via the action on the boundary $\partial \Hyp$, seen as a circle with the standard round metric. 
\end{rem}

\begin{defn}[Bounded Fuchsian group]\label{defn:Fuchsian} A \emph{bounded Fuchsian group} is a non-virtually-cyclic, finitely generated, discrete subgroup $G$ of $\operatorname{Isom}(\Hyp)$ without parabolics. Such a group acts properly and cocompactly on the \emph{convex hull} $D_G$ of the limit points $\Lambda(G)\subseteq \partial\Hyp$ for the action. If moreover $D_G=\Hyp$ then $G$ is called a \emph{cocompact} Fuchsian group. A bounded Fuchsian group is \emph{orientation-preserving}
\end{defn}

\begin{rem} The term ``bounded Fuchsian group" has appeared in the literature with different meanings. For example, in \cite{DahmaniGuirardel_Isoprob}, Dahmani and Guirardel do not require the action on $\Hyp$ to be faithful, but only to have finite kernel; their bounded Fuchsian groups are therefore \emph{finite extensions} of groups satisfying \Cref{defn:Fuchsian}. Their notion is taken from Bowditch's paper \cite{Bowditch_JSJ}, which further makes a distinction between bounded Fuchsian groups and cocompact ones. Classically (see e.g. \cite{katok_fuchsian}) Fuchsian groups are commonly understood as subgroups of $\operatorname{Isom}^+(\Hyp)$, so we prefer to exclude non-faithful actions. As a consequence, with our notation Fuchsian groups have no non-trivial finite normal subgroups: see \Cref{lem:nofinite_normal_sgr_orbifold}. On the other hand, we do include orientation-reversing isometries: our cocompact Fuchsian groups are also known as \emph{non-Euclidean crystallographic groups} \cite{NEC}. Finally, in contrast with Bowditch's notation, we consider cocompact Fuchsian group as special examples of bounded ones, to simplify the notation later.
\end{rem}

\noindent The following is well-known, but we provide a proof for completeness.
\begin{lem}\label{lem:nofinite_normal_sgr_orbifold}
Let $G$ be a bounded Fuchsian group. Then $G$ is Gromov-hyperbolic. Moreover, it has no non-trivial finite normal subgroups, and in particular it has trivial centre.
\end{lem}
\begin{proof} Gromov-hyperbolicity follows from the fact that $G$ acts properly and cocompactly on $D_G$, which is a convex subspace of $\Hyp$ and is therefore hyperbolic. Therefore, since $G$ is not virtually cyclic, it admits a maximal finite normal subgroup $K$, which is the kernel of the action of $G$ on $\Lambda(G)=\partial D_G$ (see e.g. \cite[Corollary 1.2]{baik-jang}, which is far more general). Now $D_G$ has more than two points at infinity, as otherwise $G$ would be virtually cyclic; hence every element $k\in K$ is a finite-order isometry of $\Hyp$ fixing at least three points at infinity, and is therefore trivial (this is because any finite-order isometry fixes a point in $\Hyp$, and an element of $O(2)$ fixing three points on the circle is trivial). Hence $G$ has no non-trivial finite normal subgroups; in particular, $G$ has trivial centre, as the latter is a finite normal subgroup (see e.g. the far more general \cite[Theorem 6.14.(a)]{DGO}).
\end{proof}

\subsubsection{Orbifold fundamental groups}
\noindent We now specialise the more general definition of an \emph{orbifold} to our setting. To each bounded Fuchsian group we shall associate two orbifolds, obtained as quotients of $D_G$ and $\Hyp$, respectively. The distinction will be relevant as the former will have boundary components which will not appear in the latter.

\begin{defn}[Orbifold]\label{defn:orbifold}
    Let $G$ be a bounded Fuchsian group. The associated \emph{compact orbifold} is the quotient $O=D_G/G$, while we denote the full quotient $O'=\Hyp/G$ as the \emph{punctured orbifold} associated to $G$. The group $G$ is called the \emph{orbifold fundamental group} of both $O$ and $O'$, denoted $\piorb{O}=\piorb{{O'}}$. We say that $O$ and $O'$ are \emph{orientable} if $G$ is orientation-preserving.
\end{defn}
\begin{rem}
    One could define orbifolds (and their fundamental groups) in greater generality as topological spaces locally modelled after quotients of the Euclidean space by finite subgroups of isometries, without passing through the universal cover: see, e.g. \cite[Section 3.6]{Martelli}.
\end{rem}

\begin{defn}[Singular locus]
    The \emph{singular locus} is the image in $O$ (resp. $O'$) of the points of $\Hyp$ with non-trivial $G$-stabiliser. Given a point $x$ in the singular locus, and a preimage $p\in \Hyp$ of $x$, we say that:
    \begin{itemize}
        \item $x$ is a \emph{cone point} if $\Stab{G}{p}\cong \Z/m$, generated by a rotation; $m$ is called the \emph{weight} of $x$, and by construction the total angle around $x$ is $2\pi/m$.
        \item $x$ is a \emph{mirror point} if $\Stab{G}{p}\cong \Z/2$, generated by a reflection.
        \item $x$ is a \emph{corner reflector} if $\Stab{G}{p}$ is a dihedral group, generated by two reflections along lines intersecting at $p$.
    \end{itemize}
    The quotient map $D_G\to O$ (resp. $\Hyp\to O'$) is a branched covering, branching over the singular locus. See \Cref{fig:hyperelliptic} for an example of an orbifold with cone points, and \Cref{fig:666} for an example of an orbifold with mirrors and corner reflectors.
\end{defn}

\begin{figure}[htp]
    \centering
    \includegraphics[width=\linewidth]{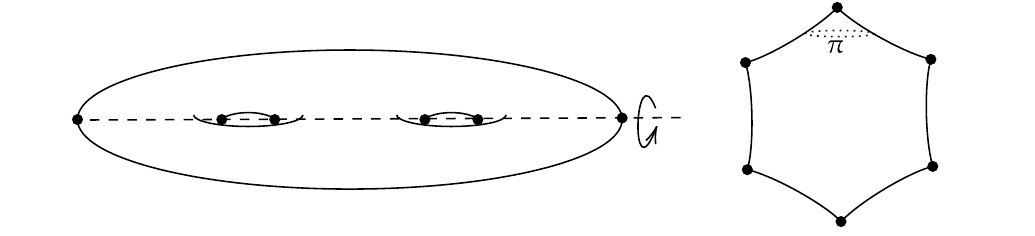}
    \caption{Consider the genus two closed surface $\Sigma_2$,  and choose a hyperbolic metric making the rotation $\rho$ by an angle of $\pi$ around the dashed axes an isometry (one can either explicitly construct such a metric, or invoke Nielsen's realization theorem, see e.g. \cite[Theorem 7.1]{FM}). The quotient $\Sigma_2/\langle \rho\rangle$ is a sphere with six cone points, each of weight 2, and indeed the total angle around each cone point is $\pi$.}
    \label{fig:hyperelliptic}
\end{figure}

\begin{figure}[htp]
    \centering
    \includegraphics[width=0.3\linewidth]{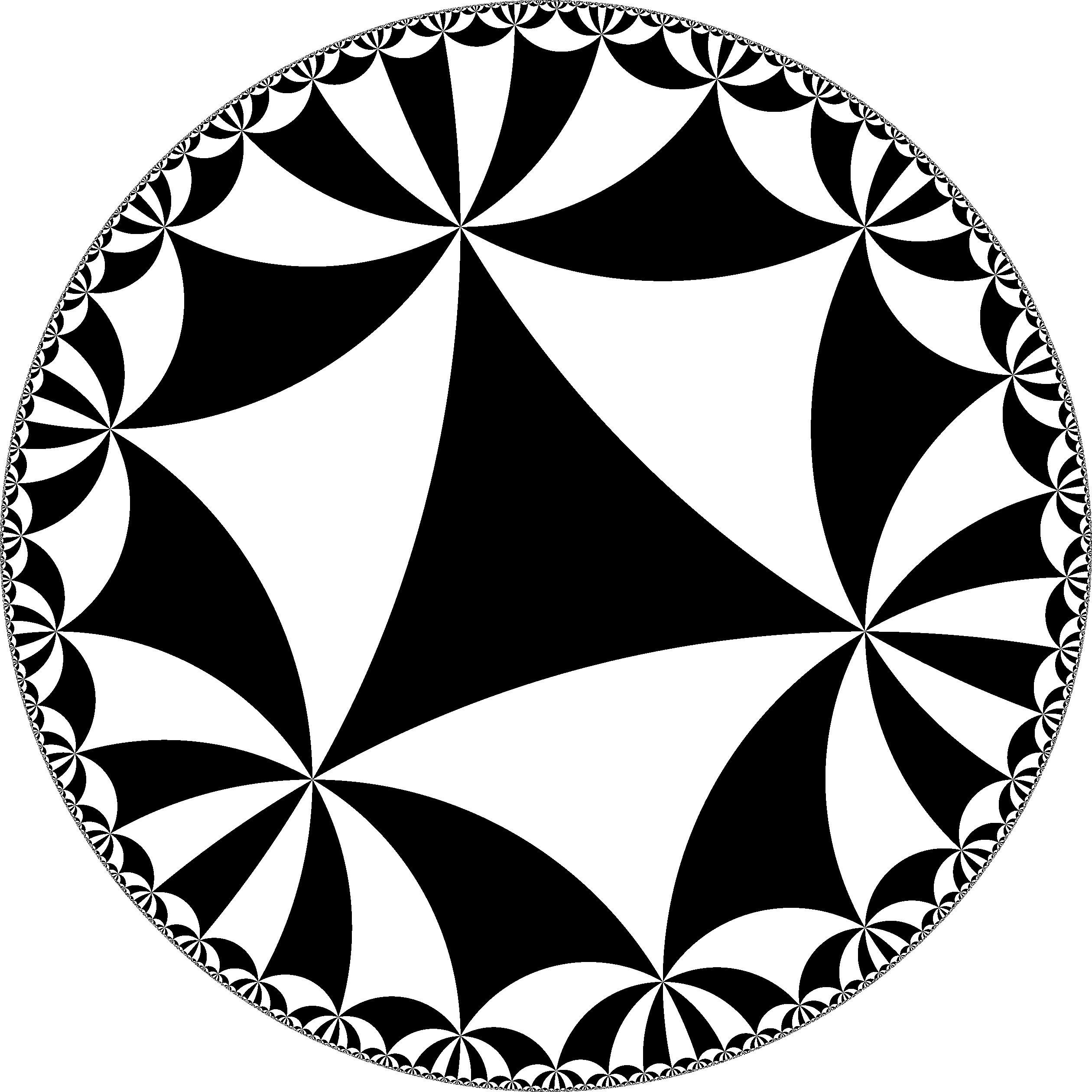}
    \caption{The above image is work of A. Sherwood, whose source can be found on \href{https://commons.wikimedia.org/w/index.php?curid=24949653}{Wikipedia}. The $(6,6,6)$ triangle group (that is, the Coxeter group on a triangle with all edge labels $6$) acts on $\Hyp$ by reflections along the sides of a hyperbolic triangle $T$ with all angles $\pi/6$, whose translates cover the space. The quotient orbifold is isomorphic to $T$ itself, and has three mirror segments intersecting at three corner reflectors of angles $\pi/6$. The stabiliser of any vertex $v$ in $T$ is a copy of the dihedral group $D_6$, generated by the reflections along the sides of $T$ meeting at $v$.}
    \label{fig:666}
\end{figure}

\begin{defn}[Orbifold boundary]
    The \emph{orbifold boundary} of the compact orbifold $O$, denoted $\partial O$, is the image of the boundary of $D_G$. Given a boundary line $\ell\subset \partial D_G$, the stabiliser of $\ell$ in $G$ is isomorphic to either $\Z$, generated by a loxodromic along $\ell$, or the infinite Dihedral group $D_{\infty}\cong \Z/2 *\Z/2$, generated by two reflections along lines orthogonal to $\ell$. Hence a boundary component $C$ of $O$ is either a circle or a segment, respectively. A \emph{peripheral subgroup} associated to $C$ is the $\Z$ or $D_{\infty}$ subgroup of $\piorb{O}$ stabilising a line in the preimage of $C$; any two such subgroups are conjugate, so with a little abuse of notation we speak of ``the" peripheral subgroup associated to $C$.
\end{defn}

\begin{rem}[Topological boundary versus orbifold boundary]\label{rem:orbifold_boundary}
    Notice that the orbifold boundary might be a proper subset of the topological boundary of $O$. Indeed, the latter might also contain \emph{mirrors}, which do not come from the boundary of $D_G$: see \Cref{fig:hexagon} for an example. We therefore stress that, in this paper $\partial O$ will always denote the \emph{orbifold} boundary.
\end{rem}

\begin{figure}[htp]
    \centering
    \includegraphics[width=\linewidth]{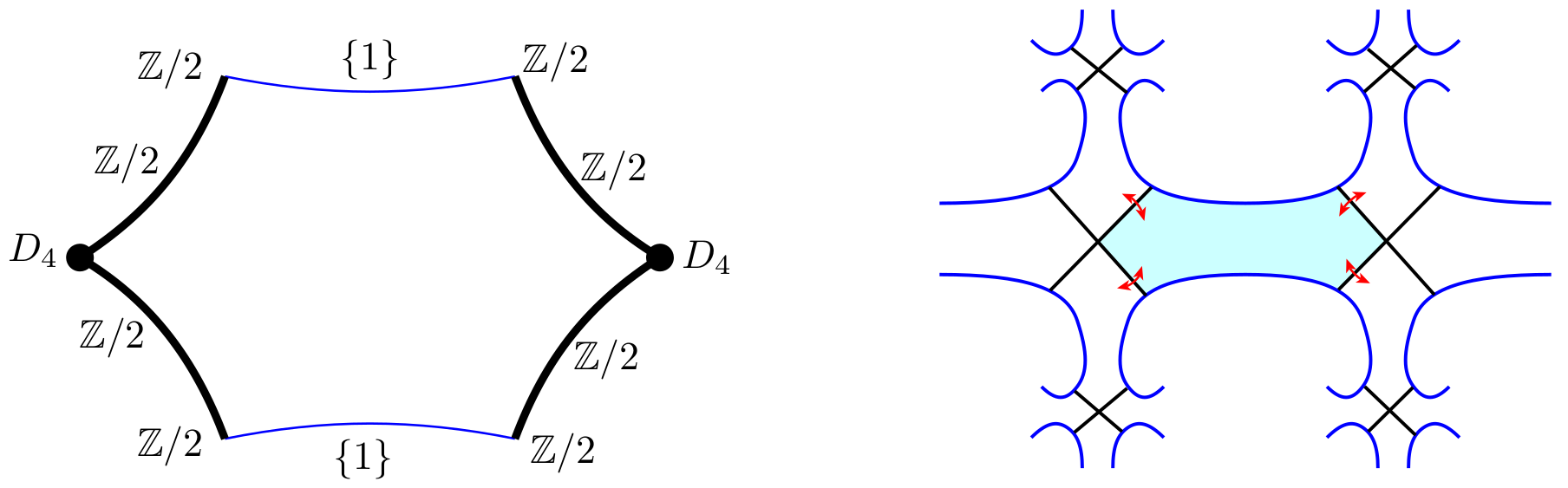}
    \caption{The above is \cite[Figure 5]{Guirardel_Levitt}. On the right, the bounded Fuchsian group $D_4* D_4$ acts cocompactly on a fattened tree, with each generator acting as the reflection along one of the black lines. On the right, the quotient compact orbifold is a hexagon with two orbifold boundary components (in blue), two corner reflectors with stabilisers $D_4$ (the black dots), and four mirror segments (in black). The orbifold boundary is thus smaller than the whole topological boundary of the hexagon. The peripheral subgroup associated to each boundary component is a copy of $D_{\infty}$, generated by the reflections along two consecutive black lines.}
    \label{fig:hexagon}
\end{figure}

\begin{rem}\label{rem:funnels}
    The punctured orbifold $O'$ is obtained from $O$ by attaching to each orbifold boundary component $C\subseteq \partial O$ a copy of $C\times [0,\infty)$, which is the image in $O'$ of a half-space of $\Hyp$ cut out by a boundary line of $D_G$. If $C$ is a circle, this is topologically the same as capping the circle with a once-punctured disk, justifying the terminology ``punctured orbifold". We stress that, while $O$ and $O'$ are homotopy equivalent and have the same (orbifold) fundamental group, they are not homeomorphic; this distinction will later be reflected also in their mapping class groups. 
\end{rem}

\subsubsection{Mapping class groups for orbifolds of conical type}
For the rest of the Section, unless otherwise stated, we shall work under the following assumption:

\begin{defn}
    A bounded Fuchsian group is \emph{of conical type} if it does not contain reflections. As a consequence, the associated orbifolds $O$ and $O'$ contain no mirrors, and in particular the topological boundary of $O$ coincide with its orbifold boundary. Notice that orientation-preserving bounded Fuchsian groups are automatically of conical type.
\end{defn}

\begin{defn}[Orbifold mapping class group]\label{defn:orbifold_MCG}
Let $G$ be a bounded Fuchsian group of conical type, and let $P$ be either the compact or the punctured orbifold associated to $G$. Let $\operatorname{Homeo}(P)$ be the group of self-homeomorphisms of $P$ which fix the orbifold boundary $\partial P$ pointwise, and which map each cone point to a cone point of the same weight. The \emph{extended orbifold mapping class group} of $P$, denoted by $\EMCG{P}$, is the quotient of $\operatorname{Homeo}(P)$ by isotopies relative to $\partial P$ and to the cone points. If $P$ is orientable, we define the \emph{mapping class group} of $P$, denoted by $\MCG{P}$, in the same way, but restricting to orientation-preserving homeomorphism. If $P$ has punctures, the \emph{pure} (extended) mapping class group, denoted $\PMCG{P}$ (resp. $\EPMCG{P}$), is the finite index subgroup of $\MCG{P}$ (resp. $\EMCG{P}$) of mapping classes which fix each puncture. 
\end{defn}

\begin{rem}
    Notice that $\MCG{O}$ and $\MCG{O'}$ are not isomorphic: the former must fix each boundary component pointwise, while the latter can permute punctures. The precise relationship between them is clarified in \Cref{lem:capping_ext} below. 
\end{rem}

\noindent Let $P$ be one of the two orbifolds associated to a bounded Fuchsian group of conical type, and denote by $S$ the surface obtained from $P$ by removing cone points, which are therefore replaced by punctures. The following obvious lemma relates the mapping class groups of $S$ and $P$:

\begin{lem}\label{lem:just_surface_mcg}
$\EMCG{P}$ is the finite-index subgroup of $\EMCG{S}$. More precisely, an element $f\in\EMCG{S}$ belongs to $\EMCG{P}$ if and only if it setwise preserves each collection of punctures coming from cone points of the same weight.
\end{lem}

\noindent The surface $S$ is of finite-type (i.e. it is obtained from a compact surface after removing finitely many punctures). By the classification of surfaces of finite-type (see, e.g. \cite{Classification_of_cpt_surf}), $S$ is homeomorphic to either:
\begin{itemize}
    \item $S^b_{g,p}$, an orientable surface of genus $g$ with $p$ punctures and $b$ boundaries;
    \item $N^b_{g,p}$, a non-orientable surface obtained from the connected sum of $g$ projective planes $\mathbb{RP}^2$ by removing $p$ punctures and $b$ open disks.
\end{itemize}
For later purposes, we note that the mapping class group of $N^b_{g,p}$ can be realised as a centraliser inside the mapping class group of an orientable surface:
\begin{thm}\label{thm:orientable_double_cover_centraliser}
Suppose that $(g, p, b) \not\in\{ (1, 0, 0), (2, 0, 0)\}$. The orientation double cover $S^{2b}_{g-1,2p}\to N^b_{g,p}$ induces an injective homomorphism $\MCG{N^b_{g,p}}\hookrightarrow \MCG{S^{2b}_{g-1,2p}}$ whose image consists of the orientation-preserving homeomorphisms which commute with the deck transformation. 
\end{thm}

\begin{proof} This is proven in \cite[Theorem 1]{Birman_chillingworth} for closed surfaces, in \cite[Lemma 3]{Szepietowski} for surfaces with boundaries, in \cite[Theorem 1.1]{Lima_Guaschi_Maldonado} for surfaces with punctures, and finally in \cite[Lemma 3.1]{Katayama_Kuno} for surfaces with boundaries and punctures. Explicitly, one can represent $S^{2b}_{g-1,2p}$ inside $\R^3$ in such a way that it is invariant under the antipodal map $J(x,y,z)=(-x,-y,-z)$, and then notice that $S^{2b}_{g-1,2p}\to S^{2b}_{g-1,2p}/J\cong N^b_{g,p}$ is a covering. The map $\MCG{N^b_{g,p}}\to \MCG{S^{2b}_{g-1,2p}}$ then corresponds to lifting a representative of a mapping class on $N^b_{g,p}$ to a orientation-preserving homeomorphism of $S^{2b}_{g-1,2p}$; the choice is unique as the covering has degree two and $J$ is orientation-reversing.
\end{proof}

\noindent The mapping class groups of the two orbifolds associated to a bounded Fuchsian group of conical type fit into the following central extension:
\begin{lem}[Capping extension]\label{lem:capping_ext}
Let $G$ be a bounded Fuchsian group of conical type, and let $O'=\Hyp/G$ and $O=D_G/G$ be the associated orbifolds. Let $\gamma_1,\ldots, \gamma_k$ be the boundary components of $O$, and denote by $T_i$ the Dehn twist along $\gamma_i$. The map $\MCG{O}\to \PMCG{O'}$, given by extending a mapping class on $O$ to the identity on $O'-O$, fits into a central extension
\[1\to \Z^k\to \MCG{O}\to \PMCG{O'}\to 1,\]
where $\Z^k$ is generated by $\{T_i\}_{i=1\ldots,k}$.
\end{lem}

\begin{proof} Let $S$ (resp. $S'$) be the surface obtained from $O$ (resp. $O'$) by removing the singular locus, and let $\MCG{S',*}$ be the finite-index subgroup of $\MCG{S'}$ of all mapping classes which fix each puncture coming from a boundary component of $S$, but are allowed to permute punctures coming from cone points. Since $\MCG{O}$ (resp. $\PMCG{O'}$) sits as a finite-index subgroup of $\MCG{S}$ (resp. $\MCG{S',*}$), it is enough to prove the existence of a central extension 
\[1\to \Z^k\to \MCG{S}\to \MCG{S',*}\to 1.\]
If $S$ is orientable, this is precisely the capping exact sequence from e.g. \cite[Proposition 3.19]{FM}. In the non-orientable case, suppose that $S=N^k_{g,p}$, so that $S'=N^0_{g,p+k}$. By the explicit construction in \Cref{thm:orientable_double_cover_centraliser}, we get a commutative diagram 
$$\begin{tikzcd}
   1\ar{r}& \Z^{2k}\ar{r}& \MCG{S^{2k}_{g-1,2p}}\ar{r}& \MCG{S^{0}_{g-1,2p+2k},*}\ar{r}&1\\
    1\ar{r}&\langle T_i\rangle_{i=1,\ldots, k}\ar{r}\ar[u,hook, "i"]& \MCG{N^{k}_{g,p}}\ar[u,hook]\ar{r}& \MCG{N^{0}_{g,p+k},*}\ar[u,hook]\ar{r}&1
\end{tikzcd}$$
where $\MCG{S^{0}_{g-1,2p+2k},*}$ are all mapping classes fixing the $2k$ punctures coming from the boundaries of $S^{2k}_{g-1,2p}$, and $i$ maps $\langle T_{i}\rangle$ to $\langle T_{\eta_i}T_{\eta'_i}^{-1}\rangle$, where $\eta_i,\eta'_i$ are the two boundary curves which lift $\gamma_i$. Then the conclusion follows from the fact that the sequence above is central and exact, and therefore so is the sequence below.
\end{proof}

\noindent The central extension from \Cref{lem:capping_ext} is almost never trivial (i.e., a product). Take for example $S=S^k_{g,0}$, so that $S'=S^0_{g,k}$. If the extension were trivial, the first homology group $H_1(\MCG{S},\Z)$ would contain a $\Z^k$ factor by the K\"{u}nneth formula (see e.g. \cite[Theorem 3B.5]{Hatcher_algtop}), contradicting the fact that $H_1(\MCG{S},\Z)$ is finite as long as $g\ge2$ \cite[Theorem 5.1]{Korkmaz_Abelianization}. However:

\begin{lem}\label{lem:mcg_central_and_bounded}
    The extension from \Cref{lem:capping_ext} is bounded.
\end{lem}

\begin{proof}
Since boundedness passes to sub-extensions by \Cref{cor:subext_of_bounded_is_bounded}, we can assume $O$ has no cone points, up to replacing it by the underlying $S$ with punctures replacing cone points. By the same \Cref{cor:subext_of_bounded_is_bounded}, the commutative diagram from the proof of \Cref{lem:capping_ext} lets us reduce to showing that the extension
\[\Z^{k}\to\MCG{S^{k}_{g,p}}\to \MCG{S^{0}_{g,p+k},*}\]
is bounded. For every curve $\gamma_i$ corresponding to a circular boundary component, let $\C{\gamma_i}$ be the associated \emph{annular curve graph}, as in \cite[Section 2]{masurminsky2}, which is a quasiline on which $\MCG{S^{k}_{g,p}}$ acts without inverting the ideal endpoints and where the Dehn twist around $\gamma_i$ acts loxodromically. Let $\phi_i\colon \MCG{O}\to \R$ be the Busemann quasimorphism associated to the action, as in \Cref{ex:busemann}, which we can rescale to map $T_i$ to $1$. The integer part of $\phi_i$ is therefore a quasimorphism $\psi_i\colon \MCG{O}\to \Z$ which identifies $\langle T_i\rangle$ with $\Z$. Taking the product over all boundary components yields a quasihomomorphism $\psi=(\psi_1,\ldots,\psi_n)\colon \MCG{O}\to \Z^n$ which is the identity on $\Z^n\cong\langle T_i\rangle_{i=1,\ldots,n}$, so the extension is bounded by \Cref{bounded via qm:Abelian_ker}.
\end{proof}

\noindent We finally recall how one can relate the mapping class group of a punctured orbifold $O'=\Hyp/G$ with the \emph{type-preserving} outer automorphisms of $G$: 

\begin{defn}[Type preserving automorphisms] Let $G$ be a bounded Fuchsian group. An automorphism $\phi\in\Aut{G}$ is \emph{type-preserving} if it preserves types of elements, as defined in \Cref{rem:types}, also distinguishing when a loxodromic has an axis in the interior or on the boundary of $D_G$. Equivalently, an automorphism is type-preserving if it preserves the conjugacy classes of peripheral subgroups. We denote by $\Outstar{G}$ the quotient of type-preserving automorphisms of $G$ by inner automorphisms (which are clearly type-preserving, in view of the equivalent formulation). 
\end{defn}

\noindent Suppose that $G$ is of conical type, and let $O'$ be the associated punctured orbifold. We can identify $G=\piorb{O'}$ with the group of deck transformations $\Aut{\pi\colon \Hyp\to O'}$ of the branched covering: see e.g. \cite[Proposition 4.6.4]{Orbifold_choi}.

Now fix a basepoint $x_0\in O'$ which is not a cone point, and let $\widetilde{x}_0\in \Hyp$ be a lift of $x_0$. Let $S$ be the surface obtained by removing cone points from $O'$, so that the branched covering $\pi\colon \Hyp\to O'$ restricts to a regular covering over $S$. Given a mapping class $\phi\in \EMCG{O'}$, choose a representative $f\in \operatorname{Homeo}(O')$ fixing $x_0$: this can always be achieved, up to isotopy. By regular covering theory we can lift $f|_{S}$ to an homeomorphism $f'$ of $\pi^{-1}(S)$ fixing $\widetilde{x}_0$; then, arguing as in \cite{Maclachlan_OrbifoldMCG}, one can extend $f'$ to a homeomorphism $\widetilde{f}\colon \Hyp\to \Hyp$ which lifts $f$. 

For every $\gamma\in G$, the conjugation $\widetilde{f}_*(\gamma)=\widetilde{f}\gamma\widetilde{f}^{-1}\colon \Hyp\to \Hyp$ is a deck transformation for $\pi$, and is therefore an element of $G$. Hence $f$ defines an automorphism $\widetilde{f}_*\in \Aut{G}$. Changing representative $f$ for $\phi$, or changing the basepoints $x_0$ and $\widetilde{x}\in \pi^{-1}(x_0)$, gives elements of $\Aut{G}$ which differ by an inner automorphism; hence we get a well-defined map $\EMCG{O'}\to \Out{G}$, independent on the choices we made, which is easily seen to be a group homomorphism. Since elements of $\EMCG{O}$ are required to preserve weights of cone points and permute punctures, the above map lands in the type-preserving outer automorphisms of $G$. Summing up, we get a well-defined homomorphism
\[\EMCG{O'}\to \Outstar{G}.\]

\begin{rem}
    In the case where $G$ is torsion-free, so that $O'$ is a surface with punctures, the above construction recovers the usual map $\EMCG{O'}\to \Outstar{G}$ induced by the action on the fundamental group: every mapping class sends closed curves to closed curves, and therefore acts on the conjugacy classes of elements of $\pi_1(O')$ (see e.g. \cite[Section 8.1]{FM}).
\end{rem}

\noindent We now state the Dehn-Nielsen-Baer theorem in the language of our setup. 

\begin{thm}[Dehn-Nielsen-Baer]\label{thm:DNB}
Let $G$ be a bounded Fuchsian group of conical type, and let $O'=\Hyp/G$ be the associated punctured orbifold. The map $\EMCG{O'}\to \Outstar{G}$ defined above is an isomorphism.
\end{thm}

\begin{proof}
    The classical version of the Dehn-Nielsen-Baer theorem concerns \emph{closed} orientable surfaces, and its proof can be found in e.g. \cite[Chapter 8]{FM}. For orientable orbifolds, we refer to \cite{Maclachlan_OrbifoldMCG} and \cite{Zieschang}, who also treat the punctured case. These results were then generalised in \cite[Section 3]{Fujiwara} to the non-orientable setting.
\end{proof}

\begin{rem}
    Though we won't need it, we also mention that Elysia Wang recently generalised \Cref{thm:DNB}, proving an isomorphism between the mapping class group of an orientable surface \emph{with boundary} and its fundamental \emph{groupoid}, with one basepoint on each boundary component \cite{wang2026dehnnielsenbaertheoremboundedsurfaces}.
\end{rem}

\begin{rem}[Beyond conical type]
    In the literature, mapping class groups have been defined also for several orbifolds of non-conical-type. For instance, Fujiwara \cite{Fujiwara} considers orbifolds with mirrors but without corner reflectors. With his definition, mapping classes are required to preserve boundary components, mirrors, and cone points with the same weight, all setwise; in other words, his mapping class group is the same as our $\MCG{O'}$, after one conflates circular mirrors and circular boundary components, and in particular it is a finite-index subgroup of the surface obtained by removing boundaries and cone points, similarly to our \Cref{lem:just_surface_mcg}. In the same paper, Fujiwara also proved a version of the Dehn-Nielsen-Baer theorem for these orbifolds. Though Fujiwara's setting extends ours, we preferred to restrict our exposition to the minimal setting that we shall need later, in the interest of conciseness.
\end{rem}

\subsection{Bass-Serre theory}\label{subsec:BS_theory}
In this Subsection we recall the definition of a graph-of-groups decomposition and the associated Bass-Serre tree, referring to \cite{trees} or \cite[Section 16]{bogopolski} for further details.
\begin{defn}[Graph-of-groups]
A \emph{bidirected graph} $\Gamma$ consists of two sets $V,E$ together with maps
\begin{align*}
    E&\to V\times V; &\quad E&\to E;\\
    e&\to (i(e),t(e)) &\quad e&\to \overline{e}
\end{align*}
satisfying $\overline{\overline{e}}=e$, $e\neq \overline{e}$, and $i(\overline{e})=t(e)$. To parse this definition, the reader should think of $V$ as the set of vertices; $E$ as the set of oriented edges, which come in pairs; $\overline{e}$ as the opposite edge of $e$, and $i(e),t(e)$ as the initial and terminal vertices of $e$. 

A \emph{graph-of-groups} is the data of:
\begin{itemize}
    \item a finite bidirected graph $\Gamma$, with vertices $V$ and directed edges $E$;
    \item a collection of \emph{vertex groups} $\{G_v\}_{v\in V}$;
    \item for every edge $e\in E$, an \emph{edge group} $G_e$, satisfying $G_e=G_{\overline e}$, together with an injective homomorphism $\tau_e\colon G_e\to G_{t(e)}$.
\end{itemize}
For every $v\in V$ let $E_v$ be the collection of edges whose initial endpoint is $v$. With a little abuse of notation, we often implicitly identify $G_e$ with $\tau_{\overline{e}}(G_e)$ and regard it as a subgroup of $G_v$.

Let $Y$ be a spanning tree inside $\Gamma$. The \emph{fundamental group} of $\mathcal G$ based at $Y$ is the group $\pi_1(\mathcal G, Y)$ obtained from the free product
\[\left(\bigast_{v\in V}G_v\right)*\left(\bigast_{e\in E} \langle x_e\rangle\right)\]
by adding the following relations:
\begin{itemize}
    \item $x_e=x_{\overline e}^{-1}$;
    \item $x_e=1$ if $e$ is an edge of $Y$;
    \item For all $g\in G_e$, $x_e \tau_e(g) x_e^{-1}=\tau_{\overline{e}}(g)$.
\end{itemize}
By \cite[Corollary 16.7]{bogopolski}, up to isomorphism the fundamental group does not depend on the choice of the spanning tree, so we just denote it by $\pi_1(\mathcal G)$.

A group $G$ \emph{splits} over a family of subgroups $\mathcal A$ if $G$ is isomorphic to the fundamental group of a graph of groups $\mathcal G$ whose edge groups are in $\mathcal A$. In this case we say that $\mathcal G$ is a \emph{splitting} of $G$ over $\mathcal A$. A splitting is \emph{relative} to a family of subgroups $\mathcal P$ if every $P\in\mathcal P$ is conjugated into a vertex group. A splitting is \emph{trivial} if $G$ coincides with one of the vertex groups.
\end{defn}

For later purposes, we recall the following pioneering result of Stallings, which one can take as the definition of one-endedness for groups:
\begin{thm}[{\cite{Stallings}}]\label{thm:stallings}
    A finitely generated, infinite group $G$ is one-ended if and only if it does not split over finite subgroups.
\end{thm}

\begin{defn}[Bass-Serre tree]
    Given a graph of groups decomposition $\mathcal G$ of a group $G$, the associated \emph{Bass-Serre tree} is the bidirected tree whose vertices are the $G$-cosets of vertex groups, and for every edge $e=(v,w)\in E$ and $g\in G$ there is an oriented edge $gG_e=(gG_v,gx_eG_w)$. We often think of $T$ as the underlying undirected graph, forgetting edge orientations. There is a $G$-action on $T$ by simplicial isometries and without edge inversions, defined by $h\cdot gG_v=hgG_v$ for every $v\in V$ and $g,h\in G$. Notice that the stabiliser of $gG_v$ is $gG_vg^{-1}$ for every $v\in V$, and similarly $\text{Stab}({gG_e})=g\tau_{\overline{e}}(G_e)g^{-1}=gx_e\tau_{e}(G_e)x_e^{-1}g^{-1}$ for every $e\in E$.
\end{defn}
\begin{defn}
    A subgroup $H\le G$ is \emph{elliptic} in a decomposition $\mathcal G$ if it is conjugate to a subgroup of a vertex group, or equivalently if it fixes a vertex in the Bass-Serre tree.
\end{defn}
\begin{rem}[Canonical splitting]
    If vertex groups coincide with their normalisers in $G$, we can equivalently see the Bass-Serre tree as a graph whose vertices are \emph{conjugates} of the vertex groups, and where $G$ acts by \emph{conjugations}. In particular, the action factors through the central quotient $G/Z(G)$, identified with the group of inner automorphisms $\Inn{G}\le \Aut{G}$. We further say that the splitting is \emph{canonical} if the $\Inn{G}$-action on $T$ extends to an $\Aut{G}$-action by simplicial isometries; if this happens, every automorphism permutes the conjugacy classes of vertex groups.
\end{rem}

\subsection{JSJ trees of cylinders for hyperbolic groups}\label{subsec:jsj}
We conclude the background section with some properties of the \emph{$\mathcal{Z}$-JSJ tree of cylinders} of a one-ended hyperbolic group $G$, and the structure that this splitting induces on $\Out{G}$.

\begin{thm}\label{thm:JSJ}
    Let $G$ be a one-ended hyperbolic group, and let $\mathcal Z$ be the family of its virtually $\Z$ subgroups with infinite centre. There exists a canonical splitting $\mathcal G$ of $G$ over $\mathcal Z$, whose vertex groups $\{G_v\}_{v\in E(\Gamma)}$ come in three pairwise exclusive types:
\begin{enumerate}
    \item \label{item:elm_vertex_sgr} \emph{Elementary vertex subgroups}: $G_v$ is a maximal virtually $\Z$ subgroup of $G$.
    \item \label{item:QH_sgr} \emph{Quadratically hanging subgroups}: $G_v$ fits in an extension
    \[1\to K_v\to G_v\xrightarrow{\pi_v} Q_v\to 1,\]
    where $K_v$ is finite and $Q_v$ is a bounded Fuchsian group of conical type.
    The edge groups incident to $G_v$ coincide with the preimages in $G_v$ of the boundary subgroups of $Q_v$.
    \item \label{item:rigid_sgr} \emph{Rigid vertex subgroups}: $G_v$ is a non-elementary quasiconvex subgroup (not of type \ref{item:QH_sgr}), and admits no splitting over $\mathcal Z$ relative to the edge groups.
\end{enumerate}  We denote by $V_i$ for $i=1,2,3$ the disjoint subsets of $V=V(\Gamma)$ corresponding to the three types of vertices above. Furthermore, every edge of $\Gamma$ connects a vertex in $V_1$ to a vertex in $V_2\cup V_3$.
\end{thm}
\noindent We shall call the splitting from \Cref{thm:JSJ} the \emph{$\mathcal Z$-JSJ tree of cylinders} for $G$. We say that a quadratically hanging subgroup $G_v$ is \emph{orientation-preserving} if the quotient Fuchsian group $Q_v$ is orientation-preserving.

\begin{rem}
The terminology ``quadratically hanging subgroups" first appeared in work of Rips and Sela in the torsion-free case \cite{Rips_Sela,Sela}. Bowditch \cite{Bowditch_JSJ} calls these subgroups ``maximal hanging Fuchsian" (again because, with his notation, Fuchsian group can act non-faithfully on $\Hyp$). We refer to \cite[page 40]{Guirardel_Levitt} for various other names these groups have assumed.
\end{rem}

\begin{rem}[Quick prerequisites]
    For the proof of the theorem above we shall assume the reader has some familiarity with the theory of JSJ decompositions, trees of cylinders, and deformation spaces of $G$-trees, so we give a vague intuition here and refer to \cite{GL_defspace, GL_tree_cylinders, Guirardel_Levitt} for all details. A splitting $\mathcal G$ over $\mathcal  Z$ is a \emph{JSJ decomposition} if it is \emph{universally elliptic} (meaning that its edge groups are elliptic in any splitting of $G$ over $\mathcal Z$) and \emph{dominates} every other universally elliptic splitting $\mathcal G'$ over $\mathcal Z$ (meaning that its vertex stabilisers are elliptic in $\mathcal G'$). Such a decomposition exists by \cite[Theorem 2.16]{Guirardel_Levitt} (to apply the theorem, one has to consider splittings over all \emph{subgroups} of groups in $\mathcal Z$, which are either in $\mathcal Z$ themselves or finite, and then notice that $G$ is one-ended and therefore does not split over finite subgroups by \Cref{thm:stallings}). $\mathcal Z$-JSJ splittings are, in general, not unique: they are arranged in a \emph{JSJ deformation space}, an $\Aut{G}$-invariant simplicial complex akin to Outer space for $\Out{\mathbb{F}_n}$. One can nonetheless produce a canonical splitting by taking any $\mathcal Z$-JSJ decomposition and then passing to the \emph{JSJ tree of cylinders}. Roughly, one defines a \emph{cylinder} as a maximal sub-tree of the Bass-Serre tree consisting of edges with pairwise commensurable stabilisers; then the tree of cylinders is roughly obtained from the Bass-Serre tree by collapsing cylinders to points, and then connecting two cylinders if they overlap over some vertex. In general, the tree of cylinders associated to a deformation space might not belong to it, as its edge stabilisers might not fall in the original class; however, results from \cite{DahmaniGuirardel_Isoprob} will imply that, in our setting, the tree of cylinders is again a $\mathcal Z$-JSJ tree, and satisfies the properties we listed in \Cref{thm:JSJ}. 
\end{rem}

\begin{proof}[Proof of \Cref{thm:JSJ}] 
We first recall that a decomposition as in the Theorem exists, if one allows vertices in $V_2$ to be finite extensions of bounded Fuchsian groups, not necessarily of conical type. Indeed, if $G$ is a finite extension of a cocompact Fuchsian group, it is enough to consider the trivial splitting, whose unique vertex group (of type \eqref{item:QH_sgr}) is $G$ and without edge groups. For all other one-ended hyperbolic groups, the existence of a canonical splitting with the required vertex and edge groups is a result of Bowditch \cite[Theorem 0.1]{Bowditch_JSJ}, who built on work of Rips and Sela \cite{Rips_Sela,Sela}; in view of a result of Swarup \cite[Theorem 1]{Swarup}, the construction applies to all one-ended hyperbolic groups (besides finite extension of cocompact Fuchsian groups, which we already covered).

We now argue as in \cite[Section 4.2]{DahmaniGuirardel_Isoprob} to produce a new decomposition, this time over $\mathcal Z$ subgroups, whose quadratically hanging subgroups are finite extensions of fundamental groups of orbifolds of conical type. Let $\mathcal G_0$ be the decomposition we constructed above, and perform the following two operations:
\begin{itemize}
    \item For every vertex $v$ of $\mathcal G_0$ of type \ref{item:QH_sgr}, let $O$ be the associated orbifold with boundary, and let $N$ be a regular neighbourhood of the union of all boundary segments and all mirrors of $O$ which avoids cone points. Then one splits the vertex group $G_v$ over the preimages of all subgroups corresponding to $\partial N$, which are now curves in the interior of $O$. 
    \item Next, one collapses all edges of the new decomposition whose stabiliser has finite centre.
\end{itemize}
The resulting decomposition $\mathcal G_1$ is now a $\mathcal Z$-JSJ decomposition by \cite[Proposition 4.5]{DahmaniGuirardel_Isoprob}; furthermore, vertex groups coming from the original quadratically hanging subgroups are now finite extensions of bounded Fuchsian groups of conical type, and the incident edge groups correspond to the preimages of boundary subgroups.

$\mathcal G_1$ might not be canonical any more, but this can be fixed by passing to the tree of cylinders. Let $T_1$ be the tree corresponding to the splitting $\mathcal{G}_1$, let $T$ be its tree of cylinders, and denote by $\mathcal{G}$ be the splitting associated to the latter. We claim that $\mathcal G$ satisfies the requirements of \Cref{thm:JSJ}. Firstly, $\mathcal G$ is again a $\mathcal Z$-JSJ splitting by \cite[Lemma 2.32]{DahmaniGuirardel_Isoprob}, and it is canonical by \cite[Proposition 2.33]{DahmaniGuirardel_Isoprob}, which is an instance of \cite[Corollary 7.4]{Guirardel_Levitt}. By construction of $T$, vertices of $\mathcal G$ come in two types:
\begin{itemize}
    \item Those corresponding to cylinders of $T_1$, whose stabilisers correspond to commensurators of edge groups of $\mathcal G_1$. In particular, such a vertex group is the maximal virtually cyclic subgroup of $G$ containing some edge group of $\mathcal G_1$.
    \item Those corresponding to vertices of $T_1$ with non-elementary stabiliser, which belong to more than two cylinders. Notice that the stabiliser of such a vertex is the same in $\mathcal G_1$ and $\mathcal G$, and it is quasiconvex by \cite[Theorem 1.2]{Bowditch_JSJ}.
\end{itemize}
By \cite[Theorem 6.5]{Guirardel_Levitt}, vertices of the second type can be further subdivided into two sub-types:
\begin{itemize}
    \item Quadratically hanging subgroups, whose underlying orbifold is of conical type and whose incident edge groups are the preimages of boundary subgroups; 
    \item Rigid subgroups, which admit no splitting relative to the adjacent edge groups. 
\end{itemize}
This yields the required partition of the vertices of $\mathcal G$ into $V_1$, $V_2$ and $V_3$. Finally, by construction of the tree of cylinders, an edge $e$ of $\mathcal G$ always connects a cylinder $Y$ with a vertex with non-elementary stabiliser, proving the ``furthermore" part of the statement.
\end{proof}

\noindent Before moving on, we point out that edge groups are maximal inside non-elementary vertex groups:
\begin{lem}\label{jsj:centralisers} If $v\in V_2\cap V_3$ and $e\in E_v$, $G_e$ is maximal among virtually cyclic subgroups of $G_v$. In particular, the centraliser $Z_{G_v}(G_e)$ coincides with the centre $Z(G_e)$.
\end{lem}

\begin{proof}
The lemma is implicit in work of Levitt (see e.g. \cite[page 13]{Levitt_authyp}), but we prove it for completeness. Suppose that $e=(v,w)$ is an edge, where $w\in V_1$ by the ``furthermore" part of \Cref{thm:JSJ}. Recall that, with the notation from \Cref{subsec:BS_theory}, we are implicitly identifying $G_e$ with the subgroup $\tau_{\overline e}(G_e)=x_e\tau_{e}(G_e)x_e^{-1}\le G_v$. By \Cref{thm:JSJ}, $x_eG_wx_e^{-1}$ is a maximal virtually cyclic subgroup of $G$ containing the virtually $\Z$ subgroup $\tau_{\overline e}(G_e)$, and is unique with these properties by e.g. \cite[Lemma 6.5]{DGO}. Hence, if $\tau_{\overline e}(G_e)\leq H$ for some virtually cyclic subgroup $H\le G_v$, then $H\le x_eG_wx_e^{-1}$, and therefore \[H=H\cap x_eG_wx_e^{-1}\le \text{Stab}(G_v)\cap \text{Stab}({x_eG_w})=\text{Stab}({G_e})=\tau_{\overline e}(G_e).\] 
The second statement of \Cref{jsj:centralisers} follows from the first: since centralisers of infinite order elements in hyperbolic groups are virtually cyclic (by e.g. \cite[Corollary III.$\Gamma$.3.10]{BridsonHaefliger}), $Z_{G_v}(G_e)$ is contained inside $G_e$ by maximality, and therefore coincides with $Z(G_e)$. 
\end{proof}

\subsubsection{Relative automorphism groups of vertex groups}

\noindent We now describe $\Out{G_v}$ when $v\in V_2$. For simplicity, we drop the indices and consider any extension
\[1\to K\to E\xrightarrow{\pi} Q\to 1,\]
where $K$ is finite and $Q$ is a bounded Fuchsian group. Notice that, if $K\le P$ for some finite normal subgroup $P$, then the image of $P$ in $Q$ is a finite normal subgroup, and is therefore trivial by \Cref{lem:nofinite_normal_sgr_orbifold}. Hence $K$ is the maximal finite normal subgroup of $E$, and is therefore \emph{characteristic}, meaning that it is preserved by all automorphisms of $E$. In turn, this also means that every automorphism $\phi\in \Aut{E}$ induces an automorphism $\overline\phi\in \Aut{Q}$, defined as $\overline\phi(\pi(g))=\pi(\phi(g))$. We therefore get a map
\[\rho\colon \Aut{E}\to \Aut{Q}.\]
\begin{lem}\label{lem:rho_map}
    $\rho$ has finite kernel and finite-index image.
\end{lem}

\begin{proof}
    If $\phi\in \ker(\rho)$, then for every $g\in E$ we have $\phi(g)=g k^\phi(g)$ for some $k^\phi(g)\in K$. Notice that, for $g,h\in E$, 
    \[gh k^\phi(gh)=\phi(gh)=\phi(g)\phi(h)=g k^\phi(g)h k^\phi(h);\]
    hence \begin{equation}\label{eq:1-cocycle}
        k^\phi(gh)=h^{-1}k^\phi(g)hk^\phi(h).
    \end{equation}
    In other words, the map $k^\phi\colon E\to K$ is a \emph{1-cocycle} with respect to the $E$-action on $K$ by conjugation. Notice that a $1$-cocycle is uniquely determined by the values it takes on a generating set for $E$, in view of \Cref{eq:1-cocycle}; hence, since $E$ is finitely generated and $K$ is finite, there are finitely many 1-cocycles, so $\ker(\rho)$ is finite. 

    We now prove that the image of $\rho$ has finite index, following \cite[Section 2.2]{DahmaniGuirardel_Isoprob}. We say that two extensions $[E_1],[E_2]$ of $Q$ with kernel $K$ are equivalent if there exists a commutative diagram

    \[\begin{tikzcd}
        1\ar{r}&K\ar{r}\ar[d,"\text{id}"]&E_1\ar{r}\ar[d,"\sim"]&Q\ar{r}\ar[d,"\text{id}"]& 1\\
        1\ar{r}&K\ar{r}&E_2\ar{r}&Q\ar{r}& 1
    \end{tikzcd}\]
    Let $\mathcal E(Q,K)$ be the set of equivalence classes of extensions of $Q$ with kernel $K$. Since $Q$ is hyperbolic, it is finitely presented, and therefore $\mathcal E(Q,K)$ is finite by e.g. \cite[Lemma 2.3]{DahmaniGuirardel_Isoprob}. Note that $\Aut{Q}$ acts on $\mathcal{E}(Q,K)$ as follows: if $\alpha\in \Aut{Q}$ and $[E]\colon 1\to K\to E\xrightarrow{\pi} Q\to 1$ is an extension, then $\alpha\cdot [E]$ is the extension $ 1\to K\to E\xrightarrow{\alpha\circ \pi} Q\to 1$. The new extension is equivalent to the original one if and only if $\alpha$ lifts to an automorphism of $E$; therefore, since $\mathcal{E}(Q,K)$ is finite, a finite-index subgroup of $\Aut{Q}$ lifts to $\Aut{E}$, thus proving that $\rho$ has finite-index image.
\end{proof}

\noindent Let $\Aut{Q}^{\ell}$ be the image of $\rho$. By the above lemma we have a finite extension
\[1\to \ker(\rho)\to \Aut{E}\xrightarrow{\rho} \Aut{Q}^{\ell}\to 1.\]
If we take the quotient by the inner automorphisms of $E$, we get that $\Out{E}$ maps with finite kernel onto $\Aut{Q}^{\ell}/\rho(\Inn{E})=\Aut{Q}^{\ell}/\Inn{Q}$, where we used that every inner automorphism of $Q$ lifts to an inner automorphism of $E$. Summarising, we have:
\begin{cor}\label{cor:out_of_QH}
    Let $1\to K\to E\to Q\to 1$ be a finite extension of a bounded Fuchsian group $Q$. The map $\overline{\rho}\colon \Out{E}\to \Out{Q}$, mapping every outer automorphism of $E$ to the induced outer automorphism of $Q$, has finite kernel and finite-index image.
\end{cor}

\noindent Though we won't need this later, for completeness we include a (partial) description of $\Out{G_v}$ for vertices of type \eqref{item:elm_vertex_sgr} and \eqref{item:rigid_sgr}. Firstly, the next lemma, which follows from results in the literature, implies that vertex groups of type \eqref{item:elm_vertex_sgr} have finite outer automorphism groups:
\begin{lem}\label{lem:finite_out} A virtually cyclic group has finite outer automorphism group.
\end{lem}

\begin{proof}
    Such a group $G$ surjects onto either $\Z$ or the infinite dihedral group $D_{\infty}=\Z/2 *\Z/2$ by \cite[Lemma 3.2]{Macpherson}. In the former situation, $\Aut{G}$ is finite by \cite[Theorem 1]{alperin}. In the latter, $\Aut{G}$ is virtually cyclic by \cite[Theorem 3.4]{pettet}; furthermore $G$ has finite centre, since it is a finite extension of the centreless group $D_{\infty}$, so the quotient $\Out{G}=\Aut{G}/(G/Z(G))$ is finite. 
\end{proof}

\noindent We finally consider vertex groups of type \eqref{item:rigid_sgr}, which are quasiconvex in $G$ by \Cref{thm:JSJ} and therefore hyperbolic. Paulin's theorem states that a one-ended hyperbolic group has infinite outer automorphism group if and only if it splits over virtually $\Z$ groups \cite{Paulin}, see also \cite[Theorem 1.4]{Levitt_authyp} for a more precise statement. This does not directly apply to vertices of type \eqref{item:rigid_sgr}, as they admit no splitting \emph{relative to the edge groups} but might split over virtually cyclic subgroups in a different way. However, combining Rips machinery \cite{Bestvina-Feighn} with a relative version of Paulin's theorem (see \cite[page 13]{Levitt_authyp}) yields the following:
\begin{lem}
    If $v\in V_3$, the subgroup of $\Out{G_v}$ which preserves the conjugacy classes of edge groups is finite.
\end{lem}

\section{Out is virtually a HHG\ldots}\label{sec:virtual_HHG}
This section is devoted to the proof of Theorems~\ref{thmintro:virtual_bounded},~\ref{thmintro:virtual_HHG}, and~\ref{thmintro:out_cocompact_is_HHG} from the introduction. We start by recalling a construction of Levitt \cite{Levitt_authyp}, which abstracts the algebraic properties of the capping extension. This can be used to express a finite-index subgroup of $\Out{G}$ as a bounded central extension whose base is a product of (virtual) orbifold mapping class groups, thus proving \Cref{thmintro:virtual_bounded}. In \Cref{subsec:virtual_HHG_of_Out} we then leverage the interaction between hierarchical hyperbolicity and central extensions to prove Theorems~\ref{thmintro:virtual_HHG} and~\ref{thmintro:out_cocompact_is_HHG}: under the assumption that all quadratically hanging subgroups are orientation-preserving, we show that the outer automorphism group of a one-ended hyperbolic group $G$ is virtually a HHG, and genuinely a HHG in the cocompact Fuchsian case.

\subsection{Algebraic mapping class group}\label{subsec:mcg_a_la_levitt}
Let $G$ be a one-ended hyperbolic group, and let $\mathcal G$ be the $\mathcal Z$-JSJ tree of cylinders from \Cref{thm:JSJ}, with underlying graph $\Gamma$. With the notation from \Cref{subsec:BS_theory} and \Cref{subsec:jsj}, denote by $\AutH{G_v}{\Gamma}$ the subgroup of $\Aut{G_v}$ consisting of automorphisms that act as conjugations on each edge group. Let $\AutbH{G_v}{\Gamma}$ be the set of tuples $(\alpha;a_{e_1},\dots,a_{e_k})\in \AutH{G_v}{\Gamma}\times G_v^{k}$ such that $\alpha$ acts as conjugation by $a_e$ on $G_e$. This set can be made into a group by setting
\[(\alpha;a_{e_1},\dots,a_{e_k})\cdot(\beta;b_{e_1},\dots,b_{e_k})\coloneqq (\alpha\beta; \alpha(b_{e_1})a_{e_1},\dots;\alpha(b_{e_k})a_{e_k}).\]
The projection on the first coordinate gives an epimorphism $\AutbH{G_v}{\Gamma}\to \AutH{G_v}{\Gamma}$. Furthermore, if an element of $\AutbH{G_v}{\Gamma}$ maps to the identity in $\AutH{G}{\Gamma}$, then it is of the form $(\text{id};a_{e_1},\dots,a_{e_k}) $, where each $a_e$ belongs to the centraliser $Z_{G_v}(G_e)$. Hence we have the following short exact sequence:
\[1\to \prod_{e\in E_v}Z_{G_v}(G_{e})\to \AutbH{G_v}{\Gamma}\to \AutH{G_v}{\Gamma}\to 1.\]
\begin{lem}[{\cite[Lemma 4.1]{Levitt_authyp}}]\label{lem:when_autLevitt_is_central}
    If $Z_{G_v}(G_e)\le G_e$ for all $e\in E_v$, then the extension is central.
\end{lem}

\begin{proof}
    For simplicity suppose there is a single incident edge. Let $(\text{id}; g)$ be in the image of $Z_{G_v}(G_e)$ and let $(\alpha;a)\in \AutbH{G_v}{\Gamma}$. Notice that, since $g\in Z_{G_v}(G_e)\le G_e$, $\alpha(g)=aga^{-1}$. Hence
    \[(\alpha;a)(\text{id};g)=(\alpha;\alpha(g)a)=(\alpha;aga^{-1}a)=(\alpha;ag)=(\text{id};g)(\alpha;a),\]
    so the extension is central.
\end{proof}

\noindent There is a homomorphism $G_v\to \AutH{G_v}{\Gamma}$ mapping $g$ to $(c_g;g,\ldots,g)$, where $c_g$ is the conjugation by $g$. Let $\PMCG{G_v}$ and $\MCGb{G_v}$ be the quotients of $\AutH{G_v}{\Gamma}$ and $\AutbH{G_v}{\Gamma}$ by the image of $G_v$, respectively. One then gets the exact sequence
\begin{equation}\label{MCG_sequence}
    1\to \frac{\prod_{e\in E_v} Z_{G_v}(G_{e})}{Z(G_v)}\to \MCGb{G_v}\to \PMCG{G_v}\to 1,
\end{equation}
which is central if $Z_{G_v}(G_e)\le G_e$ for all $e\in E_v$. For every $e\in E_v$ and every $g\in Z_{G_v}(G_e)$, the image of $g$ in $\MCGb{G_v}$ is called the \emph{algebraic Dehn twist} by $g$ along $e$ near $v$.

\begin{rem}
   In \cite{GL_McCool}, $\PMCG{G_v}$ has appeared under the name of \emph{generalised McCool subgroup}, following McCool's work on outer automorphisms of free groups fixing a given collection of conjugacy classes \cite{McCool}. Thus the main content of the Dehn-Nielsen-Baer \Cref{thm:DNB} is that the pure mapping class group of a punctured surface is a McCool subgroup of its (free) fundamental group, namely the subgroup fixing the conjugacy classes generated by loops surrounding the punctures. We shall expand on this in \Cref{lem:surface_case} below.
\end{rem}

\noindent Levitt \cite[Proposition 4.2]{Levitt_authyp} proved the existence of an ``extension" map 
\begin{equation}\label{eq:lambda_Levitt}
\prod_{v\in V}\MCGb{G_v}\xrightarrow{\lambda} \Out{G},
\end{equation}
which is roughly constructed as follows: given $(\alpha;g_{e_1},\ldots g_{e_k})\in \AutH{G_v}{\Gamma}$, one can extend $\alpha$ to each adjacent vertex group $G_{i({e_i})}$ as the conjugation by $g_{e_i}$, and then propagate this process across the whole Bass-Serre tree. The map $\lambda$ has finite-index image, and the kernel $N$ lies in $\prod_{v\in V} \frac{\prod_{e\in E_v} Z_{G_v}(G_e)}{Z(G_v)}$, so it corresponds to the relations between Dehn twists on different edge groups. The $\lambda$-image of $\prod_{v\in V} \frac{\prod_{e\in E_v} Z_{G_v}(G_e)}{Z(G_v)}$ is called the \emph{group of twists} and is denoted by $\mathcal{T}$.

Using the extension map, Levitt was able to present a finite-index subgroup of $\Out{G}$ as an extension:
\begin{thm}\label{thm:main_levitt}
    There is an exact sequence
    \[1\to \mathcal T\to \Outtwo{G}\to \prod_{v\in V_2}\PMCG{G_v}\to 1,\]
    where $\Outtwo{G}$ has finite index in $\Out{G}$. If $G$ is torsion-free, then $\mathcal T$ is free Abelian of rank $|E| - |V_1|$, and the extension is central.
\end{thm}
\begin{proof}
    This is precisely \cite[Theorem 5.1]{Levitt_authyp}. We point out that, while Levitt works with Bowditch's JSJ decomposition, his arguments from \cite[Section 5]{Levitt_authyp} only require the existence of a canonical decomposition satisfying a subset of the properties from \Cref{thm:JSJ}, and therefore apply to the $\mathcal Z$-JSJ tree of cylinder word by word.
\end{proof}
\noindent Unfortunately, even in the torsion-free case, it is not clear on the nose if the above extension is bounded. Therefore, in the following pages we shall go through Levitt's construction step-by-step, to prove that a (possibly even smaller) finite-index subgroup of $\Out{G}$ is indeed a bounded central extension of a product of bounded Fuchsian groups of conical type. We first specialise algebraic mapping class groups to vertices of type \eqref{item:QH_sgr}. The next Proposition relates the associated extension of $Q_v$ with the mapping class groups of the corresponding orbifolds:

\begin{prop}\label{lem:surface_case} 
Suppose $Q_v$ is a bounded Fuchsian group of conical type, and let $O$ and $O'$ the associated orbifolds. Let  $\gamma_1,\ldots, \gamma_N$ be the boundary components of $O$, let $H_i\le Q_v$ be the associated boundary subgroups, and define $\PMCG{G_v}$ and $\MCGb{Q_v}$ as above, with respect to the collection $\{H_1,\ldots,H_N\}$. Finally, for all $i$ let $T_i\in \MCG{O}$ be the Dehn twist around $\gamma_i$. Then there is an isomorphism of short exact sequences:
\[
\begin{tikzcd}
1 \arrow[r] & \langle H_i\rangle_{i=1,\ldots, N}  \arrow[r] \arrow[d, "\cong"] & \MCGb{Q_v}  \arrow[r] \arrow[d, "\cong"] &\PMCG{Q_v}   \arrow[r] \arrow[d, "\cong"] & 1 \\
1 \arrow[r] & \langle T_i\rangle_{i=1,\ldots, N} \arrow[r] & \MCG{O} \arrow[r] & \PMCG{O'}\arrow[r] & 1
\end{tikzcd}
\]
In particular, the above extension is central and bounded by \Cref{lem:mcg_central_and_bounded}.
\end{prop}
\noindent Notice that the isomorphism on the left identifies the algebraic Dehn twist by a generator of $H_i$ with the homonymous mapping class $T_i$, thus justifying the notation.

\begin{proof} We first justify that the kernel of the extension above is  $\langle H_i\rangle_{i=1,\ldots, N}$. Indeed, by \Cref{MCG_sequence} the kernel should be $\frac{\prod_{i=1}^N Z_{Q_v}(H_i)}{Z(Q_v)}$; however $Q_v$ has trivial centre by \Cref{lem:nofinite_normal_sgr_orbifold}, each $H_i$ is isomorphic to $\Z$, and it is maximal inside $Q_v$, and therefore coincide with its centraliser. 

We now produce the vertical isomorphisms. By construction, $\PMCG{Q_v}$ is precisely the subgroup of $\Outstar{Q_v}$ acting by conjugations on boundary subgroups. Then the Dehn-Nielsen-Baer isomorphism from \Cref{thm:DNB} identifies the latter subgroup with $\PMCG{O'}$, thus giving the isomorphism on the right. Next, for every $i\le N$ we can identify $H_i$ with the subgroup generated by the corresponding Dehn twist, and this gives the isomorphism on the left of the diagram. 

The Proposition then follows from the Five Lemma if we produce a map $\lambda\colon \MCGb{Q_v}\to  \MCG{O}$ making the diagram commute. Let $P$ be the orbifold obtained by attaching a torus with one hole $\Sigma_i$ to each circular boundary component of $O$. By the Seifert-van Kampen theorem (in its orbifold version, see e.g. \cite[page 307]{Thurston_geom_topol_3mfd}), the orbifold fundamental group of $P$ splits as a graph of groups whose vertex groups are $\piorb{O}$ and $N$ copies of $\pi_1(\Sigma_i)$, and whose edge groups are the $H_i$'s for $i\le N$. Now given $\overrightarrow{\alpha}=(\alpha;a_i,\ldots, a_N)$ a representative of an element in $\MCGb{Q_v}$, we define $A\in\Aut{\piorb{P}}$ by doing $\alpha$ on $\piorb{O}=Q_v$, and by conjugating each $\pi_1(\Sigma_i)$ by $a_i$. By the Dehn-Nielsen-Baer theorem for $P$, the outer automorphism class of $A$ is induced by a mapping class inside $\MCG{P}$, which is supported on $O\subseteq P$ as $A$ preserves the conjugacy class of each curve on $\Sigma_i$ for all $i$. Hence we get a map $\lambda\colon \MCGb{Q_v}\to  \MCG{O}$, which is a homomorphism by construction. 

We are left to check that the diagram commutes. For the left square, an element $a_i\in Z(H_i)$ is mapped to $\overrightarrow{\text{id}}=(\text{id};1,\ldots,1,a_i,1,\ldots 1)$, whose image under $\lambda$ is precisely the Dehn Twist around $\gamma_i$.  Furthermore, given $\overrightarrow{\alpha}\in \MCGb{Q_v}$, the restriction of $\lambda(\overrightarrow{\alpha})$ to $O'$, seen as the interior of $O$, is precisely the mapping class realising $\alpha\in \PMCG{Q_v}$ via the Dehn-Nielsen-Baer isomorphism, so the square on the right also commutes.
\end{proof}

\begin{cor}\label{cor:MCG_seq_bounded_V2}
    If $v\in V_2$, the extension from \Cref{MCG_sequence} is central and bounded. Moreover, $\PMCG{G_v}$ is a finite extension of a finite-index subgroup of $\PMCG{Q_v}$.
\end{cor}

\begin{proof} $Z_{G_v}(G_e)\subseteq G_e$ by \Cref{jsj:centralisers}, so the extension is central by \Cref{lem:when_autLevitt_is_central}. We therefore focus on boundedness.

    Let $\overline{\rho}\colon \Out{G_v}\to \Out{Q_v}$ be the map sending every outer automorphism of $G_v$ to the induced outer automorphism of $Q_v$, which has finite kernel and finite-index image by \Cref{cor:out_of_QH}. Let $\{H_i\}_{i=1}^N$ be the boundary subgroups of $Q_v$, and for every $i$ let $G_{e_i}\le G_v$ be the edge group mapping to $H_i$ under $\pi_v\colon G_v\to Q_v$. Let $\PMCG{Q_v}^\ell=\overline{\rho}(\Out{G_v})\cap\PMCG{Q_v}$, which has finite-index in $\PMCG{Q_v}$. Clearly $\overline{\rho} (\PMCG{G_v})$ is contained inside $\PMCG{Q_v}^\ell$, and we claim that it has finite-index. Indeed, let $\alpha\in \Aut{Q_v}$ be an automorphism whose outer class belongs to $\PMCG{Q_v}^\ell$, and let $a_1,\ldots,a_N\in Q_v$ such that $\alpha$ is the conjugation by $a_i$ on $H_i$. If $\beta\in \rho^{-1}(\alpha)$, and if we choose $g_i\in\pi_v^{-1}(a_i)$ for all $i$, then for every $p\in G_{e_i}$ there exists $k\in K_v$ such that
    \[\beta(p)=g_i pg_i^{-1}k=g_i p(g_i^{-1}kg_i) g_i^{-1}\in g_i G_{e_i}g_i^{-1}.\]
    In other words, $\PMCG{Q_v}^\ell$ is contained in the $\overline{\rho}$-image of the subgroup $\mathcal M_v$ of $\Out{G_v}$ which preserves each $G_{e_i}$ up to conjugacy; in turn, the latter is a finite-index overgroup of $\PMCG{G_v}$, as each $G_{e_i}$ is virtually cyclic and therefore has finite outer automorphism group by \Cref{lem:finite_out}. Since $\overline{\rho}(\PMCG{G_v})\le \PMCG{Q_v}^\ell\le \overline\rho(\mathcal M_v)$, this proves that $\overline\rho$ maps $\PMCG{G_v}$ to a finite-index subgroup of $\PMCG{Q_v}^\ell$, hence of $\PMCG{Q_v}$. Notice that we have already proven the ``moreover" part of the statement.

    Now consider the diagram
    \[\begin{tikzcd}
        1\ar{r}&\frac{\prod_{i=1}^N Z(G_{e_i})}{Z(G_v)}\ar{r}\ar[d, "\widehat \pi_v"]&\MCGb{G_v}\ar{r}\ar[d, "\widehat \rho"]&\PMCG{G_v}\ar{r}\ar[d, "\overline\rho"]&1\\
        1\ar{r}&\langle H_i\rangle_{i=1,\ldots, N}\ar{r}&\MCGb{Q_v}\ar{r}&\PMCG{Q_v}\ar{r}&1\\
    \end{tikzcd}\]
    The map $\widehat \rho$ is defined as follows: given $\overrightarrow{\alpha}=(\alpha; a_1,\ldots, a_N)\in \AutbH{G_v}{\Gamma}$, $\widehat \rho$ maps the outer class of $\overrightarrow{\alpha}$ to the outer class of $(\rho(\alpha); \pi_v(a_1),\ldots, \pi_v(a_N))$. The map $\widehat \pi_v$ on the left is induced by $\pi_v$; this is because $\pi_v$ surjects each $G_{e_i}$ onto $H_i$ and is trivial on the centre of $G_v$, which is contained inside $K_v$. Notice that the diagram commutes by construction. 

    We are left to prove that the top extension is bounded: we shall deduce it from the fact that the bottom one is bounded by \Cref{lem:surface_case}, with similar arguments as in \cite[Lemma 2.12]{FFMS}. For every $i\le N$ let $t_i\colon H_i\to G_{e_i}$ be a homomorphic section of $\pi_v$, and notice that $t_i(H_i)\cap Z(G_{e_i})$ is non-trivial. Since the bottom extension is bounded, \Cref{bounded via qm:Abelian_ker} provides a quasihomomorphism $\psi\colon \MCGb{Q_v}\to \prod_{i=1}^N H_i$ which is the identity on $\prod_{i=1}^N H_i$. Then the composition
    \[\MCGb{G_v}\xrightarrow{\widehat \rho} \MCGb{Q_v}\xrightarrow{\psi}\langle H_i\rangle_{i=1,\ldots, N}\xrightarrow{\frac{\prod_{i=1}^N t_i}{Z(G_v)}}\frac{\prod_{i=1}^N t_i(H_i)}{Z(G_v)}\]
    is a quasihomomorphism which is the identity on $\frac{\prod_{i=1}^N (t_i(H_i)\cap Z(G_{e_i}))}{Z(G_v)}$. A bounded modification of the above map gives a quasihomomorphism $\MCGb{G_v}\to \frac{\prod_{i=1}^N Z(G_{e_i})}{Z(G_v)}$ which is the identity on $\frac{\prod_{i=1}^N Z(G_{e_i})}{Z(G_v)}$; hence the extension is bounded, again by \Cref{bounded via qm:Abelian_ker}.
\end{proof}

\noindent We are ready to prove the following, more precise version of \Cref{thmintro:virtual_bounded}:

\begin{thm}\label{thm:M} Let $G$ be a one-ended hyperbolic group. Then $\Out{G}$ is virtually a direct product $\Z^q\times M$, where $M$ fits into a commutative diagram whose rows are bounded central extensions
\[\begin{tikzcd}
    1\ar{r}&\prod_{v\in V_2}\frac{\prod_{e\in E_v} Z(G_e)}{Z(G_v)} \ar{r}\ar[d, "\lambda", two heads] &\prod_{v\in V_2}\MCGb{G_v}\ar[r,"\pi"]\ar[d, "\lambda", two heads]&\prod_{v\in V_2}\PMCG{G_v}\ar{r}\ar[d,"\text{id}"]&1\\
    1\ar{r}&Z_s \ar{r}&M\ar[r, "\rho_1"]&\prod_{v\in V_2}\PMCG{G_v}\ar{r}&1\\
\end{tikzcd}\]
\end{thm}

\begin{proof} As we now explain, \Cref{thm:M} is implicit in the proof of \cite[Theorem 5.3]{Levitt_authyp} (which again applies verbatim to the $\mathcal Z$-JSJ tree of cylinders from \Cref{thm:JSJ}). With Levitt's notation, $M$ is the image of the restriction of the map $\lambda \colon \prod_{v\in V}\MCGb{G_v}\to \Out{G}$ from \Cref{eq:lambda_Levitt} to the mapping class groups of vertex groups of type \eqref{item:QH_sgr}. Furthermore, by \cite[Proposition 4.2]{Levitt_authyp} there exists a map $\rho_1$ from a finite-index subgroup of $\Out{G}$ to $\prod_{v\in V}\PMCG{G_v}$ which, when precomposed with $\lambda$, gives the natural projection $\pi\colon \prod_{v\in V}\MCGb{G_v}\to \prod_{v\in V}\PMCG{G_v}$. The restriction of $\pi$ to $\prod_{v\in V_2}\MCGb{G_v}$ is a central extension by \Cref{jsj:centralisers} combined with \Cref{lem:when_autLevitt_is_central}; furthermore, the image $Z_s\coloneq \lambda \left(\prod_{v\in V_2}\frac{\prod_{e\in E_v} Z(G_e)}{Z(G_v)}\right)$ lies in the centre of $M$ because $\lambda\colon \prod_{v\in V}\PMCG{G_v}\to M$ is surjective, and therefore the bottom extension is central as well.
    
    Moving to boundedness, the top extension is a direct product of bounded central extensions by \Cref{cor:MCG_seq_bounded_V2}, and is therefore bounded by \Cref{cor:prod_of_bounded_is_bounded}. The bottom extension is then bounded by \Cref{cor:quot_of_bounded_by_kernel}.
\end{proof}

\subsection{(Virtual) hierarchical hyperbolicity}\label{subsec:virtual_HHG_of_Out}
We now move to the proof of Theorems~\ref{thmintro:virtual_HHG} and~\ref{thmintro:out_cocompact_is_HHG}. To make the exposition lighter, we defer the lengthy definition of a \emph{hierarchically hyperbolic group} (HHG) to \Cref{subsec:HHS}, as we will need the full power of the hierarchical machinery only in the proof of the counterexample, \Cref{thmintro:not_HHG}. The reader can therefore blackbox the following facts about HHGs, which are the only ones we shall need in this Section.

Firstly, we recall the motivating example of a HHG:
\begin{thm}[{\cite[Theorem 11.1]{HHS_II}}]\label{thm:mcg_is_HHG}
    Let $S$ be an orientable surface of finite-type. Then $\EMCG{S}$ is a HHG. 
\end{thm}
\noindent Several group operations preserve hierarchical hyperbolicity:
\begin{lem}\label{lem:HHGpropr} The following facts hold.
    \begin{enumerate}
        \item \label{lem:HHGpropr_fi-sgr_is_HHG} A finite-index subgroup of a HHG is a HHG.
        \item \label{lem:HHGpropr_finquot_is_HHG} Let $F\le G$ be a finite normal subgroup. If $G/F$ is a HHG then $G$ is a HHG.
        \item \label{lem:HHGpropr_prod} If $G$, $H$ are HHGs, then $G\times H$ is a HHG.
        \item \label{lem:HHGpropr_ext_is_HHG} Let $1\to K\to E\to G\to 1$ be a central extension, with finitely generated kernel. If $G$ is a HHG, then $E$ is a HHG if and only if the extension is bounded.
    \end{enumerate}
\end{lem}

\begin{proof}
\Cref{lem:HHGpropr_fi-sgr_is_HHG} and \Cref{lem:HHGpropr_finquot_is_HHG}, follow from the definition of a HHG \cite[Definition 1.21]{HHS_II}; see \Cref{rem:fi_and_finite_ext_are_HHG} below for more details. Moreover, \Cref{lem:HHGpropr_prod} is  \cite[Corollary 8.28]{HHS_II}, and \Cref{lem:HHGpropr_ext_is_HHG} is \cite[Theorem 3.12]{FFMS}.
\end{proof}

\begin{rem}\label{rem:orbifold_mcg_is_HHG} Let $O$ be an \emph{orientable} orbifold of conical type. Recall from \Cref{lem:just_surface_mcg} that the mapping class group of $O$ is a finite-index subgroup of the mapping class group of the surface obtained by removing cone points, which is therefore itself orientable. Thus $\MCG{O}$ is a HHG, by combining \Cref{thm:mcg_is_HHG} and \Cref{lem:HHGpropr}.\eqref{lem:HHGpropr_fi-sgr_is_HHG}.
\end{rem}

\noindent We are finally ready to prove Theorems~\ref{thmintro:virtual_HHG} and ~\ref{thmintro:out_cocompact_is_HHG}. We first consider virtual surface groups, for which we do not need to pass to a subgroup of $\Out{G}$:
\begin{thm}\label{thm:Out(Vsurface)_is_HHG}
    Let $G$ be a finite extension of an orientation-preserving, cocompact Fuchsian group $Q$ of conical type. Then $\Out{G}$ is a HHG.
\end{thm}

\begin{proof}
    By \Cref{cor:out_of_QH}, $\Out{G}$ is a finite extension of a finite-index subgroup of $\Out{Q}$. Since $Q$ is cocompact and of conical type, by \Cref{thm:DNB} $\Out{Q}$ coincides with the mapping class group of the orbifold $\Hyp/Q$; in turn, since $Q$ is orientation-preserving, the latter is a HHG by \Cref{rem:orbifold_mcg_is_HHG}. The conclusion now follows by \Cref{lem:HHGpropr}.\eqref{lem:HHGpropr_fi-sgr_is_HHG} and \eqref{lem:HHGpropr_finquot_is_HHG}.
\end{proof}

\noindent The proof for the general case is similar, but it involves the description of $\Out{G}$ as a virtual central extension from \Cref{subsec:mcg_a_la_levitt}:
\begin{thm}\label{thm:Out(Gone_end_hyp)_is_HHG}
   Let $G$ be a one-ended hyperbolic group. If all quadratically hanging subgroups in the $\mathcal Z$-JSJ tree of cylinders are orientation-preserving, then $\Out{G}$ is virtually a HHG.
\end{thm}

\begin{proof}
    Firstly, if $v\in V_2$, then $\PMCG{G_v}$ coincides with the pure mapping class group of the underlying orbifold by \Cref{lem:surface_case}, and is therefore a HHG by \Cref{rem:orbifold_mcg_is_HHG} (here we are using that $Q_v$ is orientation-preserving and of conical type). \Cref{cor:MCG_seq_bounded_V2} then describes $\PMCG{G_v}$ as a finite extension of a finite-index subgroup of $\PMCG{G_v}$, so $\PMCG{G_v}$ is a HHG by \Cref{lem:HHGpropr}.\eqref{lem:HHGpropr_fi-sgr_is_HHG} and \eqref{lem:HHGpropr_finquot_is_HHG}. Now, the subgroup $M$ from \Cref{thm:M} is a bounded central extension of the direct product of $\PMCG{G_v}$, where  $v$ ranges in $V_2$, and is therefore a HHG by combining \Cref{lem:HHGpropr}.\eqref{lem:HHGpropr_prod} and \eqref{lem:HHGpropr_ext_is_HHG}. Finally, the same \Cref{thm:M} describes a finite-index subgroup of $\Out{G}$ as a direct product of copies of $\Z$ and $M$, so the conclusion follows from \Cref{lem:HHGpropr}.\eqref{lem:HHGpropr_prod} (and the fact that $\Z$, being hyperbolic, is also hierarchically hyperbolic). 
\end{proof}

\begin{rem}\label{rem:where_orientable}
    In the proofs of both \Cref{thm:Out(Vsurface)_is_HHG} and \Cref{thm:Out(Gone_end_hyp)_is_HHG}, the orientability assumption only appeared when we used that the mapping class group of an orientable orbifold of conical type is hierarchically hyperbolic; it could therefore be removed if one proved that the mapping class group of a \emph{non-orientable} orbifold of conical type is also a HHG.
\end{rem}

\section{\ldots but not a HHG}\label{sec:full_out_not_HHG}
In this Section we present an example of a torsion-free hyperbolic group whose full outer automorphism group fails to be a hierarchically hyperbolic group, proving \Cref{thmintro:not_HHG}. To this purpose, we first unpack the definition of a hierarchically hyperbolic group in the next subsection. 

\subsection{Background on hierarchical hyperbolicity}\label{subsec:HHS}
\begin{defn}[Hierarchically hyperbolic space]\label{defn:HHS}
Let $\delta>0$ and $X$ be a $(\delta,\delta)$--quasigeodesic space.  A \emph{hierarchically hyperbolic space (HHS) structure with constant $\delta$} for $X$ is the data of an index set $\frakS$ and a set $\{ \C{W} \colon W\in\frakS\}$ of $\delta$-hyperbolic spaces $(\C{W},\dist_W)$ such that the following axioms are satisfied.  \begin{enumerate}[label=(\arabic*{})]
    \item\textbf{(Projections.)}\label{axiom:projections} For each $W \in \frakS$, there exists a \emph{projection} $\pi_W \colon X \rightarrow 2^{\C{W}}$  that is a $(\delta,\delta)$--coarsely Lipschitz, $\delta$--coarsely onto, $\delta$--coarse map.
		
    \item \textbf{(Nesting.)} \label{axiom:nesting} If $\frakS \neq \emptyset$, then $\frakS$ is equipped with a  partial order $\nest$ and contains a unique $\nest$--maximal element, denoted by $S$. When $V\nest W$, we say $V$ is \emph{nested} in $W$.  For each $W\in\frakS$, we denote by $\frakS_W$ the set of all $V\in\frakS$ with $V\nest W$.  Moreover, for all $V,W\in\frakS$ with $V\propnest W$ there is a specified non-empty subset $\rho^V_W\subseteq \C{W}$ with $\diam(\rho^V_W)\leq \delta$.
        
    \item \textbf{(Finite complexity.)} \label{axiom:finite_complexity} Any $\nest$--chain has length at most $\delta$.
            
    \item \textbf{(Orthogonality.)} \label{axiom:orthogonal} The set $\frakS$ has a symmetric relation called \emph{orthogonality}. If $V$ and $W$ are orthogonal, we write $V\perp W$ and require that $V$ and $W$ are not $\nest$--comparable. Further, whenever $V\nest W$ and $W\perp U$, we require that $V\perp U$. We denote by $\frakS_W^\perp$ the set of all $V\in \frakS$ with $V\perp W$.

    \item \textbf{(Containers.)} \label{axiom:containers}  For each $W \in \frakS$ and $U \in \frakS_W$ with $ \frakS_W\cap \frakS_U^\perp \neq \emptyset$, there exists $Q \in\frakS_W-\{W\}$ such that $V \nest Q$ whenever $V \in\frakS_W \cap \frakS_U^\perp$.  We call $Q$ the \emph{container of $U$ in $W$}.
		
    \item \textbf{(Transversality.)}\label{axiom:transversality} If $V,W\in\frakS$ are not orthogonal and neither is nested in the other, then we say $V$ and $W$ are \emph{transverse}, denoted $V\trans W$.  Moreover, for all $V,W \in \frakS$ with $V\trans W$, there are non-empty sets $\rho^V_W\subseteq \C{W}$ and $\rho^W_V\subseteq \C{V}$, each of diameter at most $\delta$.

    \item \textbf{(Consistency.)} \label{axiom:consistency} For all $x \in X $ and $U,V,W\in\frakS$:
		\begin{itemize}
			\item  if $V\trans W$, then $\min\left\{\dist_{W}(\pi_W(x),\rho^V_W),\dist_{V}(\pi_V(x),\rho^W_V)\right\}\leq \delta$,
			\item if $U\nest V$ and either $V\propnest W$, or $V\trans W$ and $W\not\perp U$, then $\dist_W(\rho^U_W,\rho^V_W)\leq \delta$.
		\end{itemize}

\item \textbf{(Bounded geodesic image (BGI).)} \label{axiom:bounded_geodesic_image} For all  $V,W\in\frakS$ and for all $x,y \in X$, if  $V\propnest W$ and $\dist_V(\pi_V(x),\pi_V(y)) \geq \delta$, then every $\C{W}$--geodesic from $\pi_W(x)$ to $\pi_W(y)$ must intersect $ N_\delta(\rho_W^V)$.

    \item \textbf{(Large links.)} \label{axiom:large_link_lemma} For all $W\in\frakS$ and  $x,y\in X $, there exists a collection $\{V_1,\dots,V_m\}\subseteq\frakS_W -\{W\}$ such that $m\le \delta \dist_{W}(\pi_W(x),\pi_W(y))+\delta$, and for all $U\in\frakS_W - \{W\}$, either $U\nest V_i$ for some $i$, or $\dist_{U}(\pi_U(x),\pi_U(y)) \leq \delta$.

    \item \textbf{(Partial realization.)} \label{axiom:partial_realisation}  If $\{V_i\}$ is a finite collection of pairwise orthogonal elements of $\frakS$ and $p_i\in  \C{V_i}$ for each $i$, then there exists $x\in  X $ \emph{realising} the tuple $(p_i)$, meaning that, for every $i$ and every $W\in \frakS$:
		\begin{itemize}
			\item $\dist_{V_i}(\pi_{V_i}(x),p_i)\leq \delta$;
			\item if $V_i\propnest W$ or $W\trans V_i$ then $\dist_{W}(\pi_W(x),\rho^{V_i}_W)\leq \delta$.
		\end{itemize}

    \item\textbf{(Uniqueness.)} \label{axiom:uniqueness} There exists a function $\theta \colon [0,\infty) \to [0,\infty)$ so that for all $r \geq 0$, if $x,y\in X $ and $\dist_X(x,y)\geq\theta(r)$, then there exists $W\in\frakS$ such that $\dist_W(\pi_W(x),\pi_W(y))\geq r$. 
\end{enumerate}
	
    We use $\frakS$ to denote the HHS structure.  We call an element $U\in \frakS$ a \emph{domain}, the associated space $\C{U}$ its \emph{coordinate space}, and call the maps $\rho_W^V$ the \emph{relative projections} from $V$ to $W$. The quantity $\delta$ is called a \emph{hierarchy constant} for $\frakS$. We often suppress reference to the projection maps, so for every $x,y\in X $ and $U\in \frakS$ we write $\dist_U(x,y)$ to mean $\dist_U(\pi_U(x), \pi_U(y))$.
\end{defn}

\begin{defn}[Hierarchically hyperbolic group]\label{defn:HHG}
A finitely generated group $G$ is a \emph{hierarchically hyperbolic group} (HHG) if the following hold.
    \begin{enumerate}[label=(\roman*{})]
        \item\label{HHG_structure} $G$ acts metrically properly and coboundedly on a space $ X $ admitting a HHS structure $\frakS$.
        \item\label{HHG_action} There is a $\nest$--, $\perp$--, and $\trans$--preserving action of $G$ on $\frakS$ by bijections such that $\frakS$ contains finitely many $G$--orbits.
        \item\label{HHG_isometries} For each $W\in \frakS$ and $g\in G$, there exists an isometry $g_W\colon \C{W} \to \C{gW}$ satisfying the following for all $V,W\in\frakS$ and $g,h\in G$.
            \begin{itemize}
                \item The maps $(gh)_W\colon \C{W} \to \C{ghW}$ and $g_{hW}\circ h_W\colon \C{W}\to \C{hW}$ coincide.
                \item For each $x\in X$, $g_W(\pi_W(x))=\pi_{gW}(g\cdot x)$ in $\C{gW}$.
                \item If $V\trans W$ or $V\propnest W$, then $g_W(\rho^V_W)=\rho^{gV}_{gW}$ in $\C{gW}$.
            \end{itemize}
    \end{enumerate}
     We often drop the indices and denote each $g_W$ simply by $g$. When the underlying HHS is not relevant, we denote a HHG by $(G,\frakS)$.
\end{defn}

\begin{rem}[Moral compass]\label{rem:moral_compass}
    When reading the definitions above, one should keep in mind the motivating example of a HHG, which is the mapping class group of an orientable surface $S$, possibly with boundary and punctures. In this context, $X$ is the marking complex from \cite[Section 2.5]{masurminsky2}; the elements of $\frakS$ are isotopy classes of subsurfaces, with nesting given by inclusion and orthogonality corresponding to disjointness (both up to isotopy); finally, the coordinate space associated to a subsurface is the corresponding curve graph, onto which $X$ maps via the subsurface projection. The various axioms from \Cref{defn:HHS} abstract properties of curve graphs and markings from \cite{masurminsky1,masurminsky2}. 
\end{rem}

\begin{rem}\label{rem:fi_and_finite_ext_are_HHG}
    It follows from \Cref{defn:HHG} that a finite-index subgroup of a HHG, or an extension of a HHG by a finite kernel, is itself a HHG via the action on the same space. This justifies \Cref{lem:HHGpropr_fi-sgr_is_HHG} and \Cref{lem:HHGpropr_finquot_is_HHG} of \Cref{lem:HHGpropr}.
\end{rem}

\begin{rem}[HHS versus HHG]\label{rem:HHS_vs_HHG}
    If $f\colon X\to Y$ is a quasi-isometry and $(Y,\frakS)$ is a HHS, then composing the coordinate projections with $f$ gives a HHS structure on $X$. In particular, a group quasi-isometric to a HHG is a hierarchically hyperbolic \emph{space}. However, being a hierarchically hyperbolic \emph{group} is not even preserved by commensurability: for instance, the $(3,3,3)$-triangle group is a finite-index overgroup of $\Z^2$ but admits no HHG structure by \cite[Corollary 4.5]{PetytSpriano_Eyries}. This subtlety will also be reflected in the proof of \Cref{thmintro:not_HHG}.
\end{rem}

\begin{notation}\label{notation:HHG}
For the rest of the section, let $(X,\frakS)$ be a HHG structure for a group $G$, and let $x_0\in X$ be a basepoint. All definitions below are instances of more general notions, which are often stated under weaker assumptions; we refer to \cite{HHS_II,quasiflat} for further details.
\end{notation}

\noindent The following is the natural notion of quasiconvexity for subspaces of a HHS. 
\begin{defn}\label{defn:hqc} 
    A subset $Y\subseteq X$ of a HHS $(X,\frakS)$ is \emph{hierarchically quasiconvex} (hqc) if there exists a function $\kappa\colon \R_{\ge 0}\to \R_{\ge 0}$ such that:
    \begin{enumerate}
        \item For every $W\in \frakS$, $\pi_W(Y)$ is $\kappa(0)$-quasiconvex;
        \item\label{item:hqc_realiz} For every $r\ge 0$ and $x\in X$, if $\sup_{W\in \frakS}\dist_W(x,Y)\le r$, then there exists $y\in Y$ such that $\dist_X(x,y)\le \kappa(r)$.
    \end{enumerate} 
    We say that a subgroup $H\le G$ of a HHG is hierarchically quasiconvex if so is $H\cdot x_0\subseteq X$ for some (equivalently, any) basepoint $x_0$. 
\end{defn}

\noindent The next lemma follows by combining \cite[Lemma 4.9]{mangioni_amalgamation} and \cite[Lemma 1.20]{quasiflat}:
\begin{lem}\label{lem:intersection_hqc}
    The intersection of two hierarchically quasiconvex subgroups of a HHG is hierarchically quasiconvex.
\end{lem}

\begin{defn}\label{defn:rank}
    The \emph{rank} of $(X,\frakS)$ is the maximum integer $\nu\in \N$ for which there exists $\nu$ pairwise orthogonal, unbounded domains. In the setting of \Cref{notation:HHG}, the rank also coincides with the maximal $n$ for which there exists a quasi-isometrically embedded copy of $\R^n$ inside $X$.
\end{defn}

\begin{defn}\label{ex:standard_flat} Let $\mathcal{U}=\{U_1,\ldots, U_k\}\subseteq \frakS$ be pairwise orthogonal, unbounded domains. A \emph{standard $k$-flat} supported on $\mathcal{U}$ is a subspace $F\subseteq X$ such that $\pi_{U_i}(F)$ is a bi-infinite quasiline for every $i=1,\ldots,k$, while $\pi_W(F)$ is uniformly bounded for every $W\in \frakS-\mathcal{U}$. A standard $k$-flat is hierarchically quasiconvex, the second item of \Cref{defn:hqc} being true by the partial realisation axiom.
\end{defn}

\noindent The following is \cite[Proposition 2.17]{HRSS}:

\begin{prop}\label{prop:hrss} Let $(X,\frakS)$ be a HHG structure for a group $G$, let $x_0$ be a basepoint, and let $A\le G$ be a virtually $\Z^k$ subgroup for some $k\ge 1$. Then there exists $l\ge k$ and an $A$-invariant collection of pairwise orthogonal domains $\mathcal U=\{U_1,\ldots,U_l\}$ such that:
\begin{enumerate}
    \item There exists $L\ge 0$ such that $\diam_{W}(A\cdot x_0)\le L$ for all $W\not \in \mathcal U$;
    \item For $1\le i\le l$, $\pi_{U_i}(A\cdot x_0)$ is a quasiline.
\end{enumerate}
\end{prop}

\begin{cor}\label{prop:quasiflat_sgr_are_std}
    In the above setting, if $A$ is virtually $\Z^\nu$, where $\nu$ is the rank of $(X,\frakS)$, then $A\cdot x_0$ is a standard flat, and in particular it is hierarchically quasiconvex.
\end{cor}

\begin{proof}
    By \Cref{prop:hrss}, there exists $k\ge \nu$ such that $A\cdot x_0$ is contained in a standard $k$-flat $F$. Notice that we must have $k=\nu$, as the latter is defined as the maximal cardinality of a collection of pairwise orthogonal elements with unbounded coordinate spaces. Thus $A\cdot x_0$ and $F$ are both $\nu$-flats, and the former is undistorted in the latter by e.g. \cite[Corollary E]{HHP}. Finally, a quasi-isometric embedding between Euclidean spaces of the same dimension is actually a quasi-isometry by e.g. \cite[Theorem 3.8]{coarse-coHopf}, proving that $A\cdot x_0$ coarsely coincides with $F$ and is thus a standard flat. 
\end{proof}

\subsection{The counterexample}\label{sec:counterex}
We are finally ready to prove \Cref{thmintro:not_HHG}, which we restate in more details. Let $\Sigma=S_1^1$ be a torus with one open disk removed. Let $K$ be the $2$-complex obtained by gluing three copies $\Sigma_1$, $\Sigma_2$, $\Sigma_3$ along the boundary circle $C$, as in \Cref{fig:three_tori}, and let $G=\pi_1(K)$. The Seifert-van Kampen Theorem yields a splitting $\mathcal G$ of $G$ whose vertex groups are each $G_i\coloneq \pi_1(\Sigma_i)$ and $G_C\coloneq \pi_1(C)$, and edge groups identify $G_C$ with the boundary subgroup of each $G_i$. 

\begin{lem}
    With the above notation, $G$ is a one-ended, torsion-free hyperbolic group; furthermore, the decomposition $\mathcal G$ satisfies all properties from \Cref{thm:JSJ}. 
\end{lem}

\begin{proof} A finite subgroup of $G$ must fix a point on the Bass-Serre tree of $\mathcal G$; hence $G$ is torsion-free, as so are the vertex groups $G_i\cong F_2$ and $G_C\cong \Z$. Furthermore, $G$ is hyperbolic as it satisfies the requirements of \cite[Corollary 2]{amalgam_of_hyp}. By Stallings' theorem \Cref{thm:stallings}, to prove $G$ is one-ended we have to show that, whenever $G$ acts on a tree $T$ with trivial edge stabilisers, there is a global fixed point. To see this, for $1\le i<j\le 3$ let $S_{ij}=\Sigma_i\cup \Sigma_j$, so that $\pi_1(S_{ij})$ is a surface subgroup of $G$. Each $\pi_1(S_{ij})$ is one-ended, so it fixes a point $p_{ij}$ in every $G$-action on a tree with trivial edge stabilisers. Then $G_i=\pi_1(S_{ij})\cap \pi_1(S_{ik})$ fixes the geodesic $[p_{ij},p_{ik}]$; in turn, $G$ fixes the intersection of these three geodesics, which is non-empty since $T$ is a tree.

For the ``furthermore" part, the only non-trivial fact to check is that $\mathcal G$ is canonical. This is because, by \cite[Theorem 1.2]{Lafont}, every automorphism of $G$ is induced by a homeomorphism of $K$, and must therefore preserve the conjugacy classes of vertex groups. 
\end{proof}

\noindent By the above Lemma, $G$ fits the framework of \Cref{thm:Out(Gone_end_hyp)_is_HHG}: indeed, the quadratically hanging subgroups of $\mathcal G$ are precisely the fundamental groups of the $\Sigma_i$, which are orientable surface groups. Hence $\Out{G}$ is virtually a HHG. However:

\begin{thm}\label{thm:threesurfaces}
    With the above notation, $\Out{G}$ is not a HHG.
\end{thm}

\noindent The core of the proof below is, morally, that any HHG structure for $\Out{G}$ would induce a HHG structure for the subgroup generated by Dehn twists and the elements permuting the $\Sigma_i$'s; this subgroup is isomorphic to the $(3,3,3)$-triangle group, which admits no HHG structure as mentioned in \Cref{rem:HHS_vs_HHG}. We will actually need a weaker fact, as we now explain.

\begin{proof}
Towards a contradiction let $(X,\frakS)$ be a HHG structure for $\Out{G}$. For each $i$ let $T_i\in \Out{G}$ be (the outer automorphism class of) the Dehn twist of $G_i$ around $G_C$, which span a subgroup $\mathcal T$ isomorphic to $\Z^2$ as $T_1T_2T_3=1$. We first prove the following:
\begin{claim}
$\mathcal T$ is hierarchically quasiconvex.
\end{claim}
\begin{claimproof} Let $\TOut{G}$ be the finite-index subgroup of $\Out{G}$ mapping each torus to itself in an orientation-preserving way. This subgroup fits inside a central extension 
\[1\to \mathcal T\to \TOut{G}\to \prod_{i=1}^3\MCG{\Sigma_i'}\to 1,\]
where $\Sigma_i'=S_{1,1}$ is obtained from $\Sigma_i$ by gluing a once-punctured disk to the boundary curve (see the introduction of \cite{Levitt_authyp}; one can also check that $\TOut{G}$ coincides with the subgroup $\Outtwo{G}$ from \Cref{thm:main_levitt}). Since $\TOut{G}$ has finite index in $\Out{G}$, $(X,\frakS)$ is also a HHG structure for $\TOut{G}$. In particular, the above central extension is bounded by \Cref{lem:HHGpropr}.\eqref{lem:HHGpropr_ext_is_HHG}, and in turn \Cref{thm:gersten} gives that $\TOut{G}$ is quasi-isometric to the direct product $\mathcal T\times \prod_{i=1}^3\MCG{\Sigma_i'}$.

For every $i=1,\ldots, 3$ let $Z_i$ be an infinite cyclic subgroup of $\MCG{\Sigma_i'}$, and let $F\cong \Z^5$ be the preimage of $\langle Z_1,Z_2,Z_3\rangle$ inside $\TOut{G}$. Notice that the rank of $\TOut{G}$, defined as in \Cref{defn:rank}, is exactly $5$ because $\TOut{G}$ is quasi-isometric to a product of $\Z^2$ and three free groups; hence $F\cong \Z^5$ is hierarchically quasiconvex by \Cref{prop:quasiflat_sgr_are_std}.

Since $\MCG{\Sigma_i'}$ is virtually free for every $i$, we can find infinite cyclic subgroups $Z_i'\le \MCG{\Sigma_i'}$ which intersect the corresponding $Z_i$ trivially. Let $F'$ be the preimage of $\langle Z_1',Z_2',Z_3'\rangle$ inside $\TOut{G}$. Since $F'$ is hierarchically quasiconvex as well, by \Cref{lem:intersection_hqc} the intersection $\mathcal T=F\cap F'$ is hierarchically quasiconvex (in the HHG structure for $\TOut{G}$, which is the same as that for $\Out{G}$).
\end{claimproof}

By \Cref{prop:hrss}, $\mathcal T$ projects to quasilines on a collection $\mathcal L$ of at least two pairwise orthogonal domains, and to uniformly bounded sets elsewhere. Since $\mathcal T\cong \Z^2$ is hierarchically quasiconvex, $\mathcal L$ must contain exactly two elements, or it would coarsely coincide with a standard flat of higher rank. 

Now let $\sigma$ be the element of $\Out{G}$ induced by cyclically permuting the three surfaces. Notice that $\sigma T_i\sigma^{-1}=T_{i+1}$, where indices are taken modulo $3$; in particular $\sigma$ normalises $\mathcal T$, and therefore acts on $\mathcal L$. Since $\sigma$ has order $3$ while $|\mathcal L|=2$, $\sigma$ fixes each $U\in \mathcal L$; for the same reason, the action of $\sigma$ on $\partial \C U$ is trivial for every $U\in \mathcal L$, since $\C U$ is a quasiline and therefore only has two ideal endpoints. Now let $U\in \mathcal L$ be such that $T_1^2$ acts loxodromically on $U$, which exists by \cite[Theorem 3.1]{DHS_correction}, and let $\varepsilon\in \partial \C{U}$ be the attracting endpoint for this action. But then $T_2^2=\sigma T_1^2 \sigma^{-1}$ and $T_3^2=\sigma^2 T_1^2 \sigma^{-2}$ also act on $U=\sigma(U)=\sigma^2(U)$ with $\varepsilon=\sigma(\varepsilon)=\sigma^2(\varepsilon)$ as the attracting endpoint, contradicting the fact that $T_1^2T_2^2T_3^{2}=1$.
\end{proof}

\bibliography{biblio}\bibliographystyle{alpha}

\end{document}